\documentclass[a4paper,12pt]{amsart}

\usepackage{rotating,pdflscape,color,amssymb,hyperref,subfigure,psfrag,amsmath,eufrak,bbm,epsfig}
\usepackage{a4wide,amscd}
\usepackage[small]{caption}

\newcommand{\Cc}{\mc{C}}
\newcommand{\Ga}{\Gamma} 
\newcommand{\ga}{\gamma}

\newcommand{\ld}{\ldots}
\newcommand{\beg}{\begin}
\newcommand{\en}{\end}

\newcommand{\trm}{\textrm}

\newcommand{\bgt}{\begin{itemize}}
\newcommand{\ent}{\end{itemize}}
\newcommand{\ite}{\item}

\newcommand{\op}{\operatorname}
\newcommand{\eqre}{\eqref}
\newcommand{\re}{\ref}
\newcommand{\la}{\label}

\newcommand{\si}{\sigma}

\newcommand{\Lc}{\mc{L}}

\newcommand{\lan}{\langle}
\newcommand{\ran}{\rangle}

\newcommand{\Cov}{\operatorname{Cov}}

\newcommand{\ds}{\displaystyle}

\newcommand{\p}{\mathbb{P}}

\newcommand{\Tr}{\operatorname{Tr}}

\newcommand{\Ninf}{\underset{N\to\infty}{\longrightarrow}}

\newcommand{\E}{\mathbb{E}}

\newcommand{\R}{\mathbb{R}}
\newcommand{\C}{\mathbb{C}}

\newcommand{\z}{\mathbb{Z}}

\newcommand{\ud}{\mathrm{d}}

\newcommand{\pro}{probability }

\newcommand{\f}{\frac}
\newcommand{\ff}{\frac{1}}
\newcommand{\lf}{\left}
\newcommand{\ri}{\right}

\newcommand{\st}{such that }

\newcommand{\lam}{\lambda}

\newcommand{\ti}{\times}

\newcommand{\var}{\mathbb{V}\mathrm{ar}}
\newcommand{\vfi}{\varphi}
\newcommand{\ste}{\, ;\, }
\newcommand{\mc}{\mathcal }

\newcommand{\eps}{\varepsilon}

\newcommand{\A}{\mc{A}}

\newcommand{\bck}{\backslash}
\newcommand{\al}{\alpha}
\newcommand{\tta}{\theta}

\newcommand{\cvLone}{\stackrel{L^1}{{\longrightarrow}}}

\newcommand{\ovl}{\overline}

\newcommand{\bbm}{\begin{bmatrix}}
\newcommand{\ebm}{\end{bmatrix}}
\newcommand{\bes}{\begin{equation*}}
\newcommand{\ees}{\end{equation*}}
\newcommand{\be}{\begin{equation}}
\newcommand{\ee}{\end{equation}}
\newcommand{\beqy}{\begin{eqnarray}}
\newcommand{\eeqy}{\end{eqnarray}}
\newcommand{\beq}{\begin{eqnarray*}}
\newcommand{\eeq}{\end{eqnarray*}}
\newcommand{\one}{\mathbbm{1}}
\newcommand{\lto}{\longrightarrow}

\newcommand{\ie}{\emph{i.e. }}

\newcommand{\bpm}{\begin{pmatrix}}
\newcommand{\epm}{\end{pmatrix}}

\newcommand{\Lvy}{L\'evy }

\newcommand{\bpr}{\beg{proof}}
\newcommand{\epr}{\en{proof}}

\newcommand{\bet}{\beta}
\newcommand{\del}{\delta}
\newcommand{\Del}{\Delta}

\newcommand{\ra}{\rangle}
\newcommand{\ub}{\mathbf{u}}

\newcommand{\bef}{\mathbf{e}}
\newcommand{\ba}{\mathbf{a}}

\newcommand{\ka}{\kappa}

\newcommand{\tY}{\tilde{Y}}

\newtheorem{Th}{Theorem}[section]

\newtheorem{propo}[Th]{Proposition}

\newtheorem{lem}[Th]{Lemma}

\newtheorem{cor}[Th]{Corollary}
\newtheorem{Def}[Th]{Definition}

\long\def\symbolfootnote[#1]#2{\begingroup
\def\thefootnote{\fnsymbol{footnote}}\footnote[#1]{#2}\endgroup}

\newcommand{\etc}{, \dots ,}
\def\esp{\mathbb E}
\def\etc{,\ldots ,}
\def\limN{\underset{N \rightarrow \infty}\longrightarrow}
\def\Nlim{\underset{N \rightarrow \infty}\lim}
\def\eps{\varepsilon}
\def\eq{\begin{eqnarray*}}
\def\qe{\end{eqnarray*}}
\def\eqa{\begin{eqnarray}}
\def\qea{\end{eqnarray}}
\def\mbb{\mathbb}
\def\trm{\textrm}

\def\cL{{\mathcal L}}

\def\bR{{\mathbb R}}
\def\bC{{\mathbb C}}
\def\bE{{\mathbb E}}
\def\bP{{\mathbb P}}
\def\ra{{\rightarrow}}

\title[]{Central limit theorems for linear statistics of heavy tailed random matrices}
\author[Florent Benaych-Georges, Alice Guionnet, Camille Male]{Florent Benaych-Georges, Alice Guionnet, Camille Male}

\thanks{FBG: florent.benaych-georges@parisdescartes.fr, MAP 5, UMR CNRS 8145 - Universit\'e Paris Descartes, 45 rue des Saints-P\`eres 75270 Paris Cedex~6, France.\\
AG: aguionne@ens-lyon.fr, CNRS \& \'Ecole Normale Sup\'eerieure
de Lyon, Unit\'e de math\'ematiques pures et appliqu\'ees, 46 all\'ee
d'Italie, 69364 Lyon Cedex 07, France and MIT, Mathematics department, 77 Massachusetts Av, Cambridge MA 02139-4307, USA.\\
CM: male@math.univ-paris-diderot.fr, Laboratoire de probablit\'es et mod\`eles al\'eatoires, Universit\'e Paris Diderot, 175 rue Chevaleret, 75013, France.\\
Research supported by ANR-08-BLAN-0311-01, Simons Foundation and Fondation Science math\'ematique de Paris.}



\begin{document}
\maketitle

\begin{abstract}We show central limit theorems (CLT) for the linear statistics of symmetric matrices with independent heavy tailed entries, including
entries in the domain of attraction of $\alpha$-stable laws and entries with moments exploding with the dimension, as in the adjacency matrices of  Erd\"os-R\'enyi graphs.
For the second model, we also prove a central limit theorem of the moments of its empirical eigenvalues distribution. The limit  laws are  Gaussian, but unlike  the case of standard Wigner matrices, the normalization is the one of the classical CLT for independent random variables.  \end{abstract}



\section{Introduction and statement of results}

Recall that a Wigner matrix is a symmetric random matrix $A=(a_{i,j})_{i,j=1\etc N}$ such that
\begin{enumerate}
	\item[1.] the sub-diagonal entries of $ A$ are independent and identically distributed (i.i.d.),
	\item[2.] the random variables $\sqrt N a_{i,j}$ are distributed according to a measure $\mu$ that does not depend on $N$ and have all  moments finite.
\end{enumerate}
This model was introduced in 1956 by Wigner \cite{wigner} who proved the convergence of the moments
\begin{equation}\label{wigconv}\lim_{N\rightarrow \infty} \mathbb E\Big[ \frac{1}{N}\Tr(A^p) \Big]=\int x^p\f{\sqrt{4-x^2}}{2\pi} \ud x\end{equation}
when $\mu$ is centered with unit variance. Moments can be easily replaced by bounded continuous  functions  in the above convergence and  
 this convergence   holds almost surely. Assumption 2 can also be weakened to
assume only that the second  moment is  finite. The fluctuations around this limit or around the expectation were first studied by Jonsson \cite{jonsson} in the (slightly different) Wishart model, then by Pastur \emph{et al.} in \cite{KKP96},  Sinai and Soshnikov \cite{sinai} with
$p\ll N^{1/2}$ possibly  going to infinity with $N$.  Since then, a long list of further-reaching results   have been obtained:  the central limit theorem was extended to so-called matrix models where the entries interact via a potential in  \cite{johansson},  
   the set of test functions was extended and the assumptions on the entries of the Wigner matrices weakened in \cite{BaiYaoBernoulli2005,BAI2009EJP,lytova,MShcherbina11},  a more general model of band matrices was considered  in   \cite{greg-ofer} (see also \cite{lytova,baiysilver} for general covariance matrices), unitary matrices where considered in  \cite{johansson88,DiSha,soshni00,DiEv}, and Chatterjee developed a general approach to these questions in \cite{chatterjee}, under the condition that the law $\mu$ can be written as a transport of the Gaussian law. Finally, but this is not really our concern here,  the fluctuations of the trace of words in several random matrices were studied in \cite{cabanal, guionnet, mingo, maurel}.
 It turns out that in these cases
		$$\Tr(A^p)-\mathbb E\big[\Tr(A^p) \big]$$
converges towards a Gaussian variable whose covariance depends on the first  four moments of $\mu$. Moments can also be replaced by regular enough  functions and  Assumption 2 can be weakened to assume  that the fourth moment only
is finite. The latter condition is however necessary as the covariance for the limiting Gaussian depends on it.
The absence of normalization by $\sqrt{N}$ shows that the eigenvalues of $A$ fluctuate very little,
as precisely studied by Erd\"os, Schlein, Yau, Tao, Vu and their co-authors, who analyzed their rigidity in e.g. \cite{ESY2,EYY,Tao-Vu_0906.0510}.

\bigskip

In this article, we extend these results for a variation of the Wigner matrix model where Assumption 2 is removed: some entries of the matrix can be very large, e.g. when $\mu$ does not have any second moment or when it depends on $N$, with moments growing with $N$. Then, 
 Wigner's convergence theorem \eqref{wigconv} does not hold, even when moments are replaced by smooth bounded functions. The analogue of the convergence \eqref{wigconv} was 
 studied when the common law $\mu$ of the entries of $ A$ 
 belongs to the domain of attraction of an $\alpha$-stable law or $\mu$ depends on $N$ and has moments  blowing  up with $N$.
Although technical, the model introduced in Hypothesis \ref{Hyp:Model} below has the advantage of containing these two examples (for $u,v$ some sequences depending implicitly on $N$, $u\ll v$ means that $u/v\lto 0$ as $N\to\infty$).

 \beg{hyp}\la{Hyp:Model}Let, for each $N\ge 1$, $A_N=[a_{ij}]$ be an $N\ti N$ real  symmetric   random matrix whose  sub-diagonal entries are some i.i.d. copies of a random variable $a$ (depending implicitly on $N$) such that:\\
$\bullet$ The random variable $a$ can be decomposed into $a=b+c$ \st  as $N\to\infty$, 
	  \beqy
  		\la{107130bis} &&\p(c\ne 0) \ll N^{-1}\\
		\la{107131bis} &&\var(b)\ll N^{-1/2}
	\eeqy
Moreover,  if the $b_i$'s are independent copies of $b$,
\begin{equation}\label{tyu}
\lim_{K\ra\infty}\lim_{N\ra\infty} \mathbb P\left(\sum_{i=1}^N (b_i-\mathbb E(b_i))^2\ge K\right)=0\,.\end{equation}

  $\bullet$ For any $\eps>0$ independent of $N$, the random variable $a$ can be decomposed into $a=b_\eps+c_\eps$ \st 
	\be\la{limceps}
		\limsup_{N\to\infty} N\, \p(c_\eps\ne 0)\le \eps
	\ee
for all $k\ge 1$, $N\E[(b_\eps-\E b_\eps)^{2k}]$ has a finite limit $C_{\eps,k}$ as $N\to\infty$. 

$\bullet$ For $\phi_N$ the function defined on  the closure $\ovl{\C^-}$ of $\C^- := \{ \lambda \in \C \ste \Im \lambda < 0\}$ by 
	\be\la{2071216h23} 
		\phi_N(\lambda) := \E \big[ \exp( -i \lambda a^2) \big], 
	\ee we have the convergence, uniform on  compact subsets of  $\ovl{\C^-}$, 
  \be\la{2071216h33}  N(\phi_N(\lam)-1)\;\lto\;\Phi(\lam),
\ee
 for a certain function $\Phi$ defined on $\ovl{\C^-}$.\\\\
 \en{hyp}
 Examples of random matrices satisfying Hypothesis \re{Hyp:Model} are defined as follows.

\begin{Def}[Models of symmetric heavy  tailed matrices with i.i.d. sub-diagonal entries]~\label{defA}
\\Let $A=(a_{i,j})_{i,j=1\etc N}$ be a random symmetric matrix with i.i.d. sub-diagonal entries.
\begin{enumerate}	\item[1.]
	We say that $A$ is a {\bf L\'evy matrix} of parameter $\alpha$ in $]0,2[$ when $A=X/a_N $ where   the entries $x_{ij}$ of $X$  have absolute values in the domain of attraction 
	of $\al$-stable distribution, more precisely 
	\be\la{ABP09exponent}\p\left(|x_{ij}|\ge u\right)=\frac{L(u)}{u^\alpha}\ee
	with a slowly varying function $L$, and
	$$a_N=\inf\{u: P\left(|x_{ij}|\ge u\right)\le\frac{1}{N}\}$$
($a_N=\tilde{L}(N)N^{1/\al}$, with $\tilde{L}(\cdot)$ a slowly varying function).\\
	\item[2.]
		 We say that $A$ is a {\bf Wigner matrix with exploding moments} with parameter $(C_k)_{k\geq 1}$ whenever the entries of $A$ are centered, and  for any $k\geq 1$
	\be\la{1971214h}
		N\E\big[ (a_{ij})^{2k}\big]\Ninf C_k,
	\ee  with 
 for a constant $C>0$. We assume that there exists a unique measure $m$  on $\mathbb R^+$ such that for all $k\ge 0$, $C_{k+1}=\int x^k dm(x)$. 
	\end{enumerate}
\end{Def}

 \beg{lem}\la{lemma10713}Both \Lvy matrices and Wigner matrices with exploding moments satisfy Hypothesis \re{Hyp:Model}. For \Lvy matrices, the function $\Phi$ is given by formula  
 \be\la{exampleintroHTAD}\Phi(\lam)=-\sigma (i\lambda)^{\alpha/2}\ee
for some constant $\sigma\ge 0$ (in this text, as specified  in the notations paragraph at the end of this section, the power functions have a cut on $\R^-$), whereas for Wigner matrices with exploding moments, the function $\Phi$ is given by  \be\la{exampleWMWEMHTAD}\Phi(\lam)=\int \underbrace{\f{e^{-i\lam x}-1}{x}}_{:=-i\lam\trm{ for $x=0$}}\ud m(x),\ee for $m$ the unique measure on $\R_+$ with moments $\int x^k\ud m(x)=C_{k+1}$, $k\ge 0$.
\en{lem}

The proof of this lemma, and of Lemmas \ref{lemma10713LGI} and  \re{lemma10713LGI2}, which show that our hypotheses hold for both  \Lvy matrices and Wigner matrices, are given in Section \ref{part:Hyp}. 

{One can   easily see that our results also apply to complex Hermitian matrices: in this case, one only needs to require Hypothesis \re{Hyp:Model} to be satisfied by the absolute value of non diagonal entries and to have $a_{11}$ going to zero  as $N\to\infty$.}

  A L\'evy matrix whose entries are truncated in an appropriate way is a Wigner matrix with exploding moments \cite{BAGheavytails, MAL122, ZAK}. The   recentered version\footnote{The recentering has in fact asymptotically  no effect on  the spectral measure $A$ as it is a rank  one   perturbation.} of the adjacency matrix of an Erd\"os-R\'enyi graph, \ie of a matrix $A$ \st   \be\la{exampleintroErdosReniymatrixAD}\trm{$A_{ij}=1$ with probability $p/N$ and $0$ with probability $1-p/N$},\ee is also an exploding moments Wigner matrix, with $\Phi(\lambda)=p(e^{-i\lambda}-1)$ (the measure $m$ of Lemma \re{lemma10713} is $p\del_1$). In this case the fluctuations were already studied in \cite{tirozzi}. The method of \cite{tirozzi} can be adapted to study the fluctuations of linear statistics of Wigner matrices with exploding moments. Nevertheless, since we actually use Wigner matrices with exploding moments to study of the fluctuations of L\'evy matrices, it is worthwhile to study these ensembles together.
	
	 The weak convergence of the empirical eigenvalues distribution of a L\'evy matrix has been established in \cite{BAGheavytails} (see also \cite{CB, BDG, BCC}) where it was shown that for any  bounded  continuous  function $f$,
 $$\lim_{N\to\infty}\frac{1}{N}\Tr(f(A))=\int f(x)\ud\mu_\alpha(x)\qquad a.s.$$
 where $\mu_\alpha$ is a heavy tailed probability measure which   depends only  on $\alpha$. Moreover, $\mu_\alpha$ converges towards the semicircle law as $\alpha$ goes to $2$.

  The convergence in moments, in expectation and in probability, of the empirical eigenvalues distribution of a  Wigner matrix with exploding moments has been established by Zakharevich in \cite{ZAK}. In that case, moments are well defined and for any continuous bounded   function $f$,
 $$\lim_{N\to\infty}\frac{1}{N}\Tr(f(A))=\int f(x)\ud\mu_{\mathbf{C}} (x)\qquad a.s.$$
 where $\mu_{\mathbf{C}}$ is a   probability measure which   depends only  on the sequence  $\mathbf{C}:=(C_k)_{k\geq 1}$.

We shall first state the fluctuations of moments of a Wigner matrix with exploding moments around their limit, namely prove the following theorem. 

\beg{Th}[The CLT for moments of Wigner matrices with exploding moments]\label{TCLHTRM2}~ 
\\Let $A=(a_{i,j})_{i,j=1\etc N}$ be a Wigner matrix with exploding moments with  parameter $(C_k)_{k\geq 1}$. Then the process
	\be
			\bigg( \frac 1 {\sqrt N} \Tr A^K -  \esp\Big[ \frac 1 {\sqrt N} \Tr A^K \Big] \bigg)_{K\geq 1}
	\ee
converges in distribution to a centered Gaussian process.
\end{Th}
   
This theorem has been established for the slightly more restrictive model of adjacency matrices of weighted Erd\"os-R\'enyi graphs in \cite{MF91,KSV04,V04}. Our proof is based on the moment method, the covariance of the process is of combinatorial nature, given in Section \ref{Sec:CLTmoments}, Formula \eqref{eq:ZKtoZPI} and Theorem \ref{Th:CLTtraffic}. 

Note that the speed of the central limit theorem is  $N^{-1/2}$ as for independent  integrable random variables,  but  differently from what happens for   standard Wigner's matrices. This phenomenon has also already been observed for adjacency matrix of random graphs \cite{Gerard_Kim,DumiJPP} and  we will see below  that it also  holds  for L\'evy matrices. It  suggests that the repulsive interactions exhibited by the eigenvalues of most models of random matrices with lighter tails than heavy tailed matrices no longer work here. 

For L\'evy matrices, moments do not make sense 
and one should consider smooth bounded test functions. 
We start as is common in random matrix theory, with the study of 
the normalized trace of the resolvant of $A$, given for 
 $z\in\C\bck  \R$ by  
 	$$G(z):=\ff{z-A}\,.$$
	By the previous results, for both L\'evy  and exploding moments Wigner matrices, $N^{-1}\Tr G(z)$ 
converges in \pro to a deterministic limit as the  parameter $N$ tends to infinity. We study the associated fluctuations.

 In fact, even in the case of Wigner matrices with exploding moments, the CLT for moments does not imply a priori the CLT for Stieltjes functions
even though concentration inequalities hold on the right scale, see \cite{BCC2}. Indeed,  one cannot approximate  smooth bounded functions by polynomials  for the total variation norm unless
one can restrict oneself to compact subsets, a point which is not clear in this heavy tail setting.  However, with additional arguments based on martingale technology, we shall prove the following result, valid for both L\'evy matrices and Wigner matrices 
with exploding moments.

\beg{Th}[CLT for Stieltjes transforms]\la{TCLHTRM}~
Under Hypothesis \re{Hyp:Model},  where we assume additionally that for $\eps>0$, there exists 
$C_\eps>0$ such that, for any $k\geq 1$, one has
	\be\la{CondMom}
		\lim_{N\ra\infty} N\mathbb E[(b_\varepsilon-\mathbb E b_\varepsilon)^2]=C_{\eps,k}\leq C_\eps^k
	\ee
then, the process 
		\be\la{217121}
			\bigg( \ff{\sqrt{N}}\Tr G(z)-\ff{\sqrt{N}}\E\big[\Tr G(z)\big] \bigg)_{z\in \C\bck\R}
		\ee
converges in distribution (in the sense of finite marginals) to a centered complex Gaussian process $Z(z)$, $z\in \C\bck\R$, whose covariance $\Cov \big( Z(z), Z(z') \big) := \E\big[Z(z)Z(z')\big]=C(z,z') $ is given in Formulas \eqref{eq:LIntro1} \eqref{eq:LIntro2} and \eqref{eq:RhodanIntro2} below.
 \end{Th}

 Note that the uniform convergence on compact subsets, in \eqre{2071216h33}, implies that $\Phi$ is   analytic   on $\mbb C^-$ and continuous on $\ovl{\C^-}$. 
 As a consequence, we can extend the central limit theorem to a larger class of  functions, namely the Sobolev space $W^{\infty,1}_0(\R)$  of absolutely continuous functions with null limits at $\pm\infty$ and with finite total variation.
    It is easy to see  that  $W^{\infty,1}_0(\R)$ is the set of   functions $f:\R\to\R$ \st there is $g\in L^1(\R)$ \st $$\int g(t)\ud t=0\qquad\trm{ and }\qquad \forall x\in \R,\; f(x)=\int_{-\infty}^x g(t)\ud t.$$ We endow $W^{\infty,1}_0(\R)$ with the total variation norm defined at \eqre{1611131} (see also \eqre{1611132}) in Section \re{sec:concentration} of the appendix.
 It is easy to see that when restricted to $W^{\infty,1}_0(\R)$, this norm is equivalent to the one usually used on the (larger) space $W^{\infty, 1}(\R)$ (see  \cite[Sec. 8.2]{brezis}) and that the set $\Cc^2_c(\R)$ of $\Cc^2$ real valued functions with compact support on $\R$ is dense in  $W^{\infty,1}_0(\R)$.
 \begin{cor}\label{genf}
 Under the hypotheses of Theorem \ref{TCLHTRM}, the process   \begin{equation}
			\Big( Z_N(f):=\ff{\sqrt{N}}\Tr f(A)-\ff{\sqrt{N}}\E\big[\Tr f(A)\big] \ste f\in W^{\infty,1}_0(\R)\Big)
		\end{equation}
 converges in law towards a centered Gaussian process $(Z(f)\ste f\in  W^{\infty,1}_0(\R))$ with covariance given by the unique continuous extension to $W^{\infty,1}_0(\R)\ti W^{\infty,1}_0(\R)$ of the functional defined on $\Cc^2_c(\R)\ti \Cc^2_c(\R)$ by 
 	$$ C(f,g)=\mathbb E\left[  \Re\left(\int_{-\infty}^\infty \ud x \int_0^{+\infty} \ud y  \bar\partial \Psi_f(x,y) Z(x+iy) \right)
	 \Re\left(\int_{-\infty}^\infty \ud x \int_0^{+\infty} \ud y  \bar\partial \Psi_g(x,y) Z(x+iy) \right)\right]$$
where $Z$ is a centered gaussian process with covariance  $C(z,z')$ given in Theorem \ref{TCLHTRM}
and the function $\Psi_f$ is given by
$$\Psi_f(x,y)=f(x)\phi(y) +if'(x)\phi(y) y,\qquad \bar\partial\Psi(x,y)=\pi^{-1}(\partial_x+i\partial_y)\Psi(x,y)$$
 where $\phi$ is a smooth compactly supported function equal to one in the neighborhood of the of the origin.
  \end{cor}

The function $C(z,z')$ in Theorem \ref{TCLHTRM} is given by $C(z,z')= L(z,z') -L(z)L(z')$, where 
	\eqa\label{eq:LIntro1}
		L(z) & = & \int_0^{\op{sgn}\Im(z) \infty}  \frac 1 {t} \partial_{z} e^{itz+\rho_{z}(t)} \trm dt ,\\
		L(z,z') & = & \int_{0}^1 \,  \int_0^{\op{sgn}\Im(z) \infty} \int_0^{\op{sgn}\Im(z) \infty}  \frac 1 {tt'} \partial^2_{z,z'} e^{itz+it'z'+\rho_{z,z'}^u(t,t')} \trm dt \trm dt' \, \trm du.\label{eq:LIntro2}
	\qea
The maps $\rho_z $ and $\rho_{z,z'}^u$ are given by $\rho_{z,z'}^u(t,t') = \esp_{z,z'}^u\big[ \Phi(ta+t'a') \big]$, $\rho_{z}(t)=\rho_{z,z'}^u(t,0)$, where $(a,a')$ is distributed according to the non random weak limit as $N\rightarrow \infty$, $k/N \rightarrow u$ of 
	\beqy 
			  \la{eq:RhodanIntro2}
		  \frac 1 {N-1} \sum_{j=1}^{N-1} \delta_{ G_k(z)_{jj}, G'_k(z')_{jj}},
	\eeqy
where $G_k(z) = (z-A_k)^{-1}$ for $A_k$ the matrix obtained from $A$ by deleting the $k$-th row and column, and $G'_k(z') = (z'-A_k')^{-1}$ for $A_k'$ a copy of $A_k$ where the entries $(i,j)$ for $i$ or $j\geq k$ are independent of $A_k$ and the others are those of $A_k$. Also, we assume $t\Im z\ge 0$ and $t'\Im z'\ge0$ in 
\eqre{eq:RhodanIntro2}. 
\\
\par The existence of the limit 
\eqref{eq:RhodanIntro2} is a consequence of a generalized convergence in moments, namely the convergence in distribution of traffics, of $(A_k,A'_k)$ stated in \cite{MAL122}, see Lemma \ref{Lem:CVparlesTraffics}. However, under stronger assumptions, an independent proof of this convergence and an intrinsic characterization of 
$\rho_{z,z'}^u$ are provided in Theorem
\ref{Th:TixPtEqRhozz'} below.

Let us first mention that the map $\rho_z$ is the almost sure point-wise limit 
	\eqa \label{eq:RhodanIntro}
		\rho_z(t) & = & \Nlim \frac 1 {N-1} \sum_{j=1}^{N-1} \Phi\big( t G_k(z)_{jj} \big),
	\qea
 and it characterizes the limiting eigenvalues distribution of $A$. Indeed, under the assumptions of Theorem \ref{TCLHTRM}, for all $z\in \C^+$, we have the almost sure convergence 
\begin{equation}\label{Gf}\lim_{N\rightarrow \infty} \frac{1}{N} \Tr G_N(z)=  i \int_0^\infty e^{i t z + \rho_z(t)} \ud t.\end{equation}
This fact was known in the L\'evy case \cite{BAGheavytails} and is proved in greater generality in Corollary \ref{corconv1}. Let us first give a characterization of $\rho_z$.

\beg{hyp}\la{Hyp:ModelLGI}The function $\Phi$ of \eqre{2071216h33}  admits the decomposition \be\la{hypcalcconv}
		\Phi(z)=\int_0^\infty g(y) e^{i\frac{y}{z}} \ud y
	\ee
	where $g(y)$ is a function  \st for some constants $K, \ga>-1, \ka\ge 0$, we have 
	\be\la{ConditionOng}
		|g(y)| \le K \one_{y\le 1} y^\ga +K \one_{y\ge 1} y^{\ka}, \qquad \forall y>0.
	\ee
\en{hyp}

The following lemma insures us that our two main examples satisfy Hypothesis \re{Hyp:ModelLGI} (note that it is also the case when the function $\Phi(x^{-1})$ is in $L^1$ and has its Fourier transform supported by $\R_+$). 
\beg{lem}\la{lemma10713LGI}Both \Lvy matrices and Wigner matrices with exploding moments satisfy Hypothesis \re{Hyp:ModelLGI}. For \Lvy matrices, the function $g$ is  $g(y)=C_\al y^{\f{\al}{2}-1}$,  with $C_\alpha=-\sigma i^{\alpha/2}$,   whereas for Wigner matrices with exploding moments, the function $g$ is   \be\la{exampleWMWEMHTADLGI}g(y)=-\int_{\R^+} \underbrace{\f{J_1(2\sqrt{xy})}{\sqrt{xy}}}_{:=1\trm{ for }xy=0}\ud m(x),\ee for $m$ the measure on $\R_+$ of Lemma \re{lemma10713} 
and $J_1$ the Bessel function of the first kind  defined by $\ds J_1 (s) = \frac s 2 \sum_{k \geq 0} \frac { (-s^2 / 4) ^k } {k ! (k+1)!}$. 
\en{lem}
 
 \beg{Th}[A fixed point equation for $\rho_z(t)$]\la{Th:TixPtEqRhoz}
Under Hypotheses \re{Hyp:Model} and \re{Hyp:ModelLGI}, the function $(z\in \C^+, \lam\in\R^+)\mapsto  \rho_z(\lam)$ is analytic in its first  argument and continuous in the second, with non positive real part, and characterized among the set of such functions by the formula 
	\be\la{charrhoAD}
		\rho_z(\lam)=\lam\int_0^\infty g(\lam y) e^{i{y}z
 +\rho_z({y})}\ud y.
	\ee
\end{Th}
Above, the second point in  Hypotheses \re{Hyp:Model}  is not required anymore, as it served mainly to prove convergence, which is now insured by the uniqueness of limit points. 

Note   that the asymptotics of Wigner matrices with bounded moments is also described by \eqref{charrhoAD}. In this case $\Phi(\lambda) = -i \lambda$, so $g(y)=-i$ and $\rho_z(t) = -it \Nlim \frac 1 N \Tr G(z)$, which leads, by formula \eqref{Gf},
 to the classical quadratic equation
		\be\la{3113}s(z):=  \Nlim \frac 1 N \Tr G(z)  = \frac 1 {z - s(z)}.\ee

Let us now give a fixed point characterization for the function $\rho_{z,z'}?(t,t')$ of \eqre{eq:RhodanIntro2}.

\beg{hyp}\la{Hyp:ModelLGI2}The function $\Phi$ of \eqre{2071216h33}  either has the form \be\la{uniquenessassumption2}\Phi(x)=-\sigma (ix)^{\alpha/2} \ee or admits the decomposition, for $x,y$ non zero: \be\label{phias}
\Phi(x+y)=\iint_{(\R_+)^2}e^{i\frac{v}{x}+i\frac{v'}{y}} \ud\tau(v,v')+ \int e^{i\frac{v}{x}} \ud\mu(v) + \int e^{i\frac{v'}{y}} \ud\mu(v') \end{equation}	for some  complex measures $\tau,\mu$ on respectively $(\R^+)^2$  and $\R^+$ \st for all $b>0$, $\int e^{-bv}\ud|\mu|(v)$ is finite and for some constants $K>0$, $-1< \gamma \leq 0$ and $\ka \geq 0$, and		
	\eqa\label{eq:AssumpTau}
		 \frac{\ud|\tau|(v,v')}{\ud v \ud v'}\le K \big( v^\ga \one_{v\in ]0,1]} + v^{{\ka}} \one_{v\in ]1,\infty[}\big) \big( {v'}^{\ga} \one_{v'\in ]0,1]} + {v'}^{{\ka}} \one_{v'\in ]1,\infty[}\big).
	\qea\en{hyp}
	
	\beg{rmk}\la{22101318h17}{\rm For L\'evy matrices, where $\Phi(x)= C_\al x^{\alpha/2}$ with $C_\alpha=-\sigma i^{\alpha/2}$,
\eqref{phias} holds as well. Indeed,  for all $x,y\in\mathbb C^+$ (with a constant $C_\al$ that can change at every line),
\begin{eqnarray}\la{223131}
&&\Phi(x^{-1}+y^{-1})=C_\al (\frac{1}{x}+\frac{1}{y})^{\alpha/2}= C_\al \frac{1}{x^{\alpha/2}} \frac{1}{y^{\alpha/2}} (x+y)^{\alpha/2}\\
\nonumber &=&
 C_\al\int_0^\infty \ud w \int_0^\infty \ud w'\int_0^\infty \ud v
w^{\alpha/2-1} (w')^{\alpha/2-1} v^{-\alpha/2-1} e^{ i w x+iw'y}(e^{iv(x+y)}-1)
\end{eqnarray}
(where we used the formula $\ds z^{\al/2} =C_\al\int_{t=0}^{+\infty}\f{e^{itz}-1}{t^{\al/2+1}}\ud t$ for any $z\in \C^+$ and $\al\in (0,2)$, which can be proved 
with the residues formula) so that \eqre{phias}   holds with $\mu= 0$ and $\tau(v,v')$ with density with respect to Lebesgue measure given by
	\beqy\label{eq:TauLevy}
	&&C_\al\int_0^{+\infty}  u^{-\alpha/2-1} \{(v-u)^{\alpha/2-1}(v'-u)^{\alpha/2-1}\one_{0\le u\le v\wedge v'}- v^{\alpha/2-1} {v'}^{\alpha/2-1}\}\ud u.
	\eeqy
Unfortunately $\tau$ does not satisfy \eqref{eq:AssumpTau} as its density blows up at $v=v'$: we shall treat both case separately.}\en{rmk}

The following lemma insures us that our two main examples satisfy Hypothesis \re{Hyp:ModelLGI2}. 

\beg{lem}\la{lemma10713LGI2}Both \Lvy matrices and Wigner matrices with exploding moments satisfy Hypothesis \re{Hyp:ModelLGI2}. For   Wigner matrices with exploding moments, the measure  $\tau$ is given by  \be\la{tauWMWEM}
{\ud \tau(v,v')}:=\ud v\ud v'\int\f{J_1(2\sqrt{vx})J_1(2\sqrt{v'x})}{\sqrt{vv'}}\ud  m(x)\ee  for $m$ the measure on $\R_+$ of Lemma \re{lemma10713}  and the measure $\mu$ is given by \be\la{muWMWEM}
\ud \mu(v):=-\ud v\int\f{J_1(2\sqrt{vx})}{\sqrt{v}}\ud  m(x).\ee 
\en{lem}

\beg{Th}[A fixed point  system of equations for $\rho_{z,z'}^u(t,t')$]\la{Th:TixPtEqRhozz'}
Under Hypotheses \re{Hyp:Model}, \re{Hyp:ModelLGI} and \re{Hyp:ModelLGI2},  the conclusions of Theorem \re{TCLHTRM} and Corollary \ref{genf} hold and  the parameter $\rho_{z,z'}$ of \eqref{eq:LIntro2} is given by
$$\rho_{z,z'}^u(t,s)=u\rho_{z,z'}^{u,1}(t,s)+(1-u) \rho_{z,z'}^{2}(t,s)$$
where  
 $\rho_{z,z'}^{u,1}(t,s), \rho_{z,z'}^{2}(t,s)  $ are analytic functions on   $\Lambda=\{z,z':  t \Im z>0  \,, s\Im z'>0\}$
 and uniformly continuous on compacts in the variables $(t,s)$ (and $\beta$- H\"older   for $\beta>\alpha/2$ in the L\'evy matrices case, see Lemma \ref{272132}), given by \eqre{1511121h58BD13} as far as $\rho_{z,z'}^{u,2}(t,s)$ is concerned and unique solution, among such functions, of the following fixed point equation \eqre{1511121h58} as far as  $\rho_{z,z'}^{u,1}(t,s)$ is concerned:
\begin{eqnarray} \la{1511121h58}
\rho_{z,z'}^{u,1}(t,s)&=&\int_0^{\op{sgn}_z\infty}\int_0^{\op{sgn}_{z'} \infty}e^{i\frac{v}{t}z+i\frac{v'}{s}z'}  e^{u\rho_{z,z'}^{u,1}(\frac{v}{t},\frac{v'}{s})+(1-u)(\rho_z(\frac{v}{t})+\rho_{z'}(\frac{v'}{s}))}\ud\tau(v,v')\\
&&+ \int_0^{\op{sgn}_z\infty} e^{ i\frac{v}{t}}  e^{\rho_z(\frac{v}{t})} \ud\mu(v) + \int_0^{\op{sgn}_{z'}\infty} e^{ i\frac{v}{s}}  e^{\rho_{z'}(\frac{v}{s})} \ud\mu(v)\nonumber,\\
\rho_{z,z'}^{2}(t,s)&=&\int_0^{\op{sgn}_z\infty}\int_0^{\op{sgn}_{z'} \infty}e^{i\frac{v}{t}z+i\frac{v'}{s}z'}  e^{\rho_z(\frac{v}{t})+\rho_{z'}(\frac{v'}{s})}\ud\tau(v,v')\la{1511121h58BD13} \\
&&+ \int_0^{\op{sgn}_z\infty} e^{ i\frac{v}{t}}  e^{\rho_z(\frac{v}{t})} \ud\mu(v) + \int_0^{\op{sgn}_{z'}\infty} e^{ i\frac{v}{s}}  e^{\rho_{z'}(\frac{v}{s})} \ud\mu(v)\nonumber\,,
\end{eqnarray}
with the notations $\op{sgn}_{z}:=\op{sign}(\Im z)$, $\op{sgn}_{z'}:=\op{sign}(\Im z')$ and the measures $\tau$, $\mu$ defined by Hypothesis \re{Hyp:ModelLGI2} and Remark \re{22101318h17}.
\en{Th}

Let us conclude this introduction with three  remarks.
\begin{enumerate}
\item Let $A=X/a_N$ be a  \Lvy matrix as defined at Definition \re{defA} but with $\alpha=2$.
Then  using Example c) p. 44. of \cite{feller2} instead of the hypothesis made at Equation \eqre{2071216h33}, one can prove that as $N\to\infty$, the spectral measure of $A$ converges almost surely to the semi-circle law with support $[-2,2]$ (see \eqre{3113}). This result somehow ``fills the gap" between heavy-tailed matrices  and finite second moment Wigner matrices.  It allows for example to state that if $P(|X_{ij}|\ge u)\sim cu^{-2}$, with $c>0$, even though the entries of $X$ do not have any second moment, we have that the empirical spectral law of $\f{X}{\sqrt{cN\log(N)}}$ converges almost surely to the semi-circle law with support $[-2,2]$.
	\item Our results also have an application to  {\bf standard Wigner matrices} (\ie symmetric random matrices of the form $A=X/\sqrt{N}$, with $X$ having centered i.i.d. sub-diagonal entries with variance one and not depending on $N$. In this case, 
the function $\Phi$ of \eqre{2071216h33} is   linear, which implies that  $L(z,z')=L(z)L(z')$ for all $z,z'$, so that the covariance is null, \eqre{charrhoAD} is the self-consistent equation satisfied by the Stieltjes transform of the semi-circle law, namely \eqre{3113},  and Corollary \re{genf} only means that for functions $f\in \A$, we have, for the convergence in probability, \be\la{7121214h22}\Tr f(A)-\E\big[\Tr f(A)\big]=o( \sqrt{N}).\ee
 This result is new for 
  Wigner matrices whose entries have a second but not a fourth moment, \eqre{7121214h22} brings  new information. Indeed, for such matrices, which could be called      ``\emph{semi heavy-tailed random matrices}",   the convergence to the semi circle law holds (see \cite{bai-silver-book} or the remark right above that one) but   the largest eigenvalues do not tend to the upper-bound of the support of the semi-circle law, are asymptotically in the scale $N^{\f{4-\al}{2\al}}$ (with $\al\in (2,4)$ as in Equation \eqre{ABP09exponent} when such an exponent exists) and distributed according to a Poisson process (see \cite{ABP09}), and it is not clear what the rate of convergence to the semi-circle law will be. Equation \eqre{7121214h22} shows that this rate  is $\ll N^{-1/2}$.
 	\item About recentering with respect to the limit instead of the expectation, it depends on the rate of convergence in \eqref{1971214h} or in \eqref{2071216h23}. For instance, if $N \esp (a_{11}^{2k}) - C_k = o(N^{-1/2})$ for any $k\geq 1$, then 
		$$\sqrt N \bigg( \esp\Big[ \frac 1 N \Tr A^k \Big] - \Nlim \esp\Big[ \frac 1 N \Tr A^k \Big] \bigg) \limN 0,$$
but otherwise a non trivial recentering should occur. See the end of Section \ref{Sec:CLTmoments}.
\end{enumerate}

 \noindent{\bf Organization of the article:} The CLTs for moments and Stieltjes transform (Theorems \ref{TCLHTRM2} and \ref{TCLHTRM}) are proved in Sections \ref{Sec:CLTmoments} and \ref{sec:prCLTSt10713} respectively. Corollary \ref{genf}, which extends the CLTs for   functions in $W^{\infty,1}_0$, is proved in Section \ref{part:Corgenf}. Theorems \re{Th:TixPtEqRhoz} and \re{Th:TixPtEqRhozz'} about fixed point equations for the functions $\rho_z$ and $\rho_{z,z'}^u$ expressing the limit spectral distribution and  covariance   are proved in Section \re{sec:fixedpoint}.
 Section \ref{part:Hyp} is devoted to the proof of our assumptions for L\'evy matrices and Wigner matrices with exploding moments.
\\ 
\\ {\bf Notation:} In this article, the power functions   are defined on $\C\bck\R_-$ via the standard   determination of the argument  on this set taking values in $(-\pi,\pi)$. The set $\C^+$ (resp. $\C^-$) denotes the open upper (resp. lower)  half plane and for any $z\in \C$, $\op{sgn}_{z}:=\op{sign}(\Im z)$.

\section{CLT for the moments of Wigner matrices with exploding moments}\la{Sec:CLTmoments}

The goal of this section is to prove Theorem \ref{TCLHTRM2}. In order to prove the CLT for the moments of the empirical eigenvalues distribution of $A$, we use a modification of the method of moments inspired by \cite{MAL12} which consists 
of studying more general functionals of the entries of the matrix (the so-called injective moments) than only its moments. We describe this approach below.
\par Let $A$ be a Wigner matrix with exploding moments.  Let $K\geq 1$ be an integer. The normalized trace of the $K$-th power of $A$ can be expanded in the following way.
	\beq
		\frac 1 N \Tr A^K & = & \frac 1 N \sum_{i_1 \etc i_K=1}^N A(i_1, i_2) \dots A(i_{K-1}, i_K) A(i_K,i_1)\\
			& = & \sum_{\pi\in \mathcal P(K)} \underbrace{\frac 1 N \sum_{\mathbf i \in S_\pi} A(i_1, i_2) \dots A(i_{K-1}, i_K) A(i_K,i_1)}_{\tau_N^0[\pi]},
	\eeq
	
	\noindent where $\mathcal P(K)$ is the set of partitions of $\{1\etc K\}$ and $S_\pi$ is the set of multi-indices $\mathbf i = (i_1 \etc i_K)$ in $\{1\etc N\}^K$ such that $n\sim_\pi m \Leftrightarrow i_n = i_m$. 

\par We interpret $\tau_N^0$ as a functional on graphs instead of partitions. Let $\pi$ be a partition of $\{1\etc K\}$. Let $T^\pi =(V,E)$ be the undirected graph (with possibly multiple edges and loops) whose set of vertices $V$ is $\pi$ and with multi-set of edges $E$ given by: there is one edge between two blocks $V_i$ and $V_j$ of $\pi$  for each $n$ in $\{1\etc K\}$ such that $n\in V_i$ and $n+1 \in V_j$ (with notation modulo $K$). Then, one has
	\beqy\label{Def:TraceInj}
		\tau_N^0(\pi)=\tau_N^0[T^\pi] 
	\eeqy
	if for graph $T=(V,E)$, we have denoted 
	\beqy\label{Def:TraceInj2}
\tau_N^0(T) = \frac 1 N \sum_{\substack{ \phi: V \to [N] \\ \textrm{injective}}} \prod_{ e \in E} A\big( \phi(e) \big),\eeqy
where $[N] = \{1\etc N\}$ and for any edge $e= \{i,j\}$ we have denoted $ A\big( \phi(e) \big) = A\big(\phi(i), \phi(j) \big)$. There is no ambiguity in the previous definition since the matrix $A$ is symmetric.

\par In order to prove the convergence of 
		$$\big( Z_N(K) \big)_{K\geq 1} = \bigg( \frac 1{\sqrt N} \Tr A^K - \esp \Big[ \frac 1{\sqrt N} \Tr A^K\Big] \bigg)_{K\geq 1}$$
to a Gaussian process, it is sufficient to prove the convergence of 
 	\eqa\label{Def:ProcessInjectiveTr}
		 \big(Z_N(T^\pi) \big)_{\pi \in \cup_K \mathcal P(K)} :=  \bigg( \sqrt N †\Big( \tau_N^0[T^\pi] -  \esp \big[  \tau_N^0[T^\pi] \big] \Big) \bigg)_{\pi \in \cup_K \mathcal P(K)}
	\qea
to a Gaussian process, since
		\eqa \label{eq:ZKtoZPI}
			\big( Z_N(K) \big)_{K\geq 1}  =  \bigg( \sum_{\pi \in \mathcal P(K)} Z_N(T^\pi) \bigg)_{K\geq 1}.
		\qea
Before giving the proof of this fact, we recall a result from  \cite{MAL122}, namely the convergence of $\tau_N^0[T^\pi]$ for any partition $\pi$. These limits are involved in our computation of the covariance of the limiting process of $\big(Z_N(T^\pi) \big)_{\pi \in \cup_K \mathcal P(K)}$, and this convergence will be useful in the proof of the CLT for Stieltjes transforms latter.

\begin{propo}[Convergence of generalized moments]\label{Prop:CVDistrTraff}
Let $A$ be a Wigner matrix with exploding moments with parameter $(C_k)_{k\geq 1}$. For any partition $\pi$ in $\cup_K \mathcal P(K)$, with $\tau_N^0[T^\pi]$ defined in \eqref{Def:TraceInj},
	\eq
		\esp\big[ \tau_N^0[T^\pi] \big] \limN \tau^0[T^\pi]:=\left\{ \begin{array}{cc} \prod_{k\geq 1} C_k^{q_k} & \textrm{ if } T^\pi \textrm{ is a fat tree,}\\
												0 & \textrm{ otherwise,} \end{array}\right.
	\qe
	where a fat tree is a graph that becomes a tree when the multiplicity of the edges is forgotten, and for $T$ such a graph we have denoted $q_k$ the number of edges of $T$ with multiplicity $2k$.
\end{propo}

\begin{Th}[Fluctuations of generalized moments]~\label{Th:CLTtraffic}
Let $A$ be a Wigner matrix with exploding moments. Then, the process $\big(Z_N(T^\pi) \big)_{\pi \in \cup_K \mathcal P(K)}$ defined by \eqref{Def:ProcessInjectiveTr} converges to a centered Gaussian process $\big(z(T^\pi) \big)_{\pi \in \cup_K \mathcal P(K)} $ whose covariance is given by: for any $T^{\pi_1}, T^{\pi_2}$,
	\eq
		\esp\big[ z(T^{\pi_1}) z(T^{\pi_2}) \big] = \sum_{T \in \mathcal P_\sharp(T^{\pi_1}, T^{\pi_2})} \tau^0[T],
	\qe
where $\tau^0[T]$ is given by Proposition \ref{Prop:CVDistrTraff} and $\mathcal P_\sharp(T^{\pi_1}, T^{\pi_2})$ is the set of graphs obtained by considering disjoint copies of  the graphs $T^{\pi_1}$ and $T^{\pi_2}$ and gluing them by requiring
that they have at least one edge (and therefore two ``adjacent'' vertices) in common.
\end{Th}

\begin{proof}We show the convergence of joint moments of $\big(Z_N(T^\pi) \big)$. Gaussian distribution being characterized by its moments, this will prove the theorem. Let $T_1 = (V_1,E_1) \etc T_p = (V_n,E_n)$ be finite undirected graphs, each of them being of the form $T^\pi$ for a partition $\pi$. We first write
	\beq
		\esp \Big[ Z_N(T_1) \dots Z_N(T_n) \Big] & = & \frac 1{N^{\frac n 2}}  \sum_{ \substack{ \phi_1 \etc \phi_n \\ \phi_j : V_j \to [N] \textrm{ inj.}}} \underbrace{  \esp \Bigg[ \prod_{j=1}^n \bigg( \prod_{e \in E_j} A\big(\phi_j(e) \big) - \esp \Big[ \prod_{e \in E_j} A\big(\phi_j(e) \big)  \Big] \bigg) \Bigg]}_{\omega_N(\phi_1 \etc \phi_n)}\\
	   &  = & \sum_{\sigma \in \mathcal P(V_1 \etc V_n)} \frac 1 {N^{\frac n 2}} \sum_{(\phi_1 \etc \phi_n)\in S_\sigma} \omega_N(\phi_1 \etc \phi_n),
	\eeq
where 
\begin{itemize}
	\item $\mathcal P(V_1 \etc V_n)$ is the set of partitions of the disjoint union of $V_1\etc V_n$ whose blocks contain at most one element of each $V_j$,
	
	\item $S_\sigma$ is the set of families of injective maps, $\phi_j: V_j\to [N], j=1\etc n$, such that for any $v\in V_j, v' \in V_{j'}$, one has $\phi_j(v) = \phi_{j'}(v') \Leftrightarrow v \sim_{\sigma} v'$.
\end{itemize}

 First, it should be noticed that by invariance in law of $A$ by conjugacy by permutation matrices, for any $\sigma$ in $\mathcal P(V_1 \etc V_n)$ and $(\phi_1 \etc \phi_n)$ in $S_\sigma$, the quantity $\omega_N(\phi_1 \etc \phi_n)$ depends only on $\sigma$. We then denote $\omega_N(\sigma) = \omega_N(\phi_1 \etc \phi_n)$. Moreover, choosing a partition in $\mathcal P(V_1 \etc V_n)$ is equivalent to merge certain vertices of different graphs among $T_1 \etc T_n$.  We equip $\mathcal P(V_1 \etc V_n)$ with
 the edges of $T_1,\ldots, T_n$ and say that two vertices are adjacent if there is an edge between them.
 We denote by $\mathcal P_\sharp(V_1 \etc V_n)$ the subset of $\mathcal P(V_1 \etc V_n)$ such that any graph has two adjacent vertices that are merged to two adjacent vertices of an other graph. By the independence of the entries of $X$ and the centering of the components in $\omega_N$, for any $\sigma$ in $\mathcal P(V_1 \etc V_n)\setminus \mathcal P_\sharp(V_1 \etc V_n)$ one has $\omega_N(\sigma)=0$. Hence, since the cardinal of $S_\sigma$ is $\frac {N!}{(N-|\sigma|)!}$, we get
	\beqy
		\esp \Big[ Z_N(T_1) \dots Z_N(T_n) \Big] & = & \sum_{\sigma \in \mathcal P_\sharp(V_1\etc V_n)} N^{-\frac n 2} \frac { N!}{(N-|\sigma|)!} \ \omega_N(\sigma) \nonumber\\
		& = & \sum_{\sigma \in \mathcal P_\sharp(V_1\etc V_n)} N^{-\frac n 2+ |\sigma|} \ \omega_N(\sigma)  \big( 1 + O( N^{-1}) \big). \label{eq:FirstExpanMoments}
	\eeqy

 Let $\sigma$ in $\mathcal P_\sharp(V_1\etc V_n)$. We now analyze the term $\omega_N(\sigma)$. We first expand its product.
	$$\omega_N(\sigma) = \sum_{B \subset \{1\etc n\}} (-1)^{n-|B|} \esp\bigg[ \prod_{j\in B} \prod_{e \in E_j} A\big( \phi_j(e) \big)\bigg]  \times \prod_{j\notin B} \esp\bigg[\prod_{e \in E_j} A\big( \phi_j(e) \big)\bigg].$$
	
 Let $B\subset \{1\etc n\}$. Denote by $T_B$ the graph obtained by merging the vertices of $T_j$, $j\in B$ that belong to a same block of $\sigma$. For any $k\geq 1$ denote by $p_k$ the number of vertices of $T_B$ where $k$ loops are attached. For any $k\geq\ell\geq 0$, denote by $q_{k,\ell}$ the number of pair of vertices that are linked by $k$ edges in one direction and $\ell$ edges in the other. Denote by $\mu_N$ the common law of the entries of $\sqrt N A$. By independence of the entries of $A$, for any $(\phi_1 \etc \phi_n)$ in $S_\sigma$, one has
	\beqy
		\esp\bigg[ \prod_{j\in B} \prod_{e \in E_j} A\big( \phi_j(e) \big)\bigg] & = & \prod_{k\geq 1}\bigg( \frac{ \int t^k \textrm d \mu_N(t)}{N^{\frac k 2}} \bigg)^{p_k} \prod_{k,\ell\geq 0} \bigg( \frac{ \int t^{k+\ell} \textrm d \mu_N(t)}{N^{\frac {k+\ell} 2}} \bigg)^{q_{k,\ell}}\nonumber \\
		 & = & N^{-|\bar E_B|} \underbrace{ \prod_{k\geq 1}\bigg( \frac{ \int t^k \textrm d \mu_N(t)}{N^{\frac k 2-1}} \bigg)^{p_k} \prod_{k,\ell\geq 0} \bigg( \frac{ \int t^{k+\ell} \textrm d \mu_N(t)}{N^{\frac {k+\ell} 2-1}} \bigg)^{q_{k,\ell}} }_{\delta_N(B)} \label{eq:SecondExpanMoments},
	\eeqy

\noindent where $|\bar E_B|$ is the number of edges of $T_B$ once the multiplicity and the orientation of edges are forgotten. Recall assumption 
\eqref{1971214h}: for any $k\geq 1$,
	\eq
		N \esp\big[ (a_{i,j})^{2k}\big] = \esp \bigg[ \frac{ (\sqrt N a_{i,j} )^{2k}}{N^{k-1}} \bigg] = \frac{ \int{ t^{2k} \textrm d \mu_N(t)}}{N^{k-1}} \limN C_k.
	\qe

By the Cauchy-Schwarz inequality, for any $k\geq 1$,
	\eq
		\frac{ \int{ |t|^{k+1} \textrm d \mu_N(t)}}{N^{\frac{k+1}2-1}} \leq \sqrt{ \frac{ \int{ t^{2k} \textrm d \mu_N(t)}}{N^{k-1}}} \times \sqrt{ \int t^2 \textrm d \mu_N(t) } = O(1).
	\qe

Hence, since the measure $\mu_N$ is centered, the quantity $\delta_N(B)$ is bounded. Denote $T_{\sigma}$ 
the graph obtained by merging the vertices of $T_1 \etc T_n$ that belong to a same block of $\sigma$, $|\bar E_\sigma|$ 
its number of edges when orientation and multiplicity is forgotten, and by $c_\sigma$ its number of components. We obtain from \eqref{eq:FirstExpanMoments} and \eqref{eq:SecondExpanMoments}
	\beq
		\lefteqn{\esp \Big[ Z_N(T_1) \dots Z_N(T_n) \Big]}\\
		 & = & \sum_{\sigma \in \mathcal P_\sharp(V_1\etc V_n)} \sum_{B \subset \{1\etc n\}} N^{-\frac n 2 + |\sigma| - |\bar E_B| - \sum_{j\notin B} |\bar E_{\{j\}}|}  (-1)^{n-|B|} \delta_N(B)\prod_{j\notin B}\delta_N(\{j\})  \big( 1 + O( N^{-1}) \big) \\
		 & = & \sum_{\sigma \in \mathcal P_\sharp(V_1\etc V_n)} \sum_{B \subset \{1\etc n\}}  N^{c_\sigma - \frac n 2} \times  N^{|\bar E_\sigma| - |\bar E_B| - \sum_{j\notin B} |\bar E_{\{j\}}|}  \times N^{|\sigma| - c_\sigma -|\bar E_\sigma|} \\
		 & & \ \ \ \ \times (-1)^{n-|B|} \delta_N(B)\prod_{j\notin B}\delta_N(\{j\})  \times \big( 1 + O( N^{  -1}) \big) \,.
	\eeq
 A partition $\sigma \in \mathcal P_\sharp(V_1\etc V_n)$ induces a partition $\bar \sigma$ of $\{1\etc n\}$: $i\sim_{\bar \sigma} j$ if and only if $T_i$ and $T_j$ belong to a same connected component of $T_\sigma$. Denote by $\mathcal P_2(n)$ the set of pair partitions of $\{1\etc n\}$. One has
	\beqy\label{eq:Contr1}
		N^{c_\sigma - \frac n 2} = \one_{\bar \sigma \in \mathcal P_2(n)} + O(N^{-1}).
	\eeqy

 Secondly, one has $|\bar E_\sigma| - |\bar E_B| - \sum_{j\notin B} |\bar E_{\{j\}}| \leq 0$ with equality if and only if $B = \{1\etc n\}$, so that
	\beqy\label{eq:Contr2}
		N^{|\bar E_\sigma| - |\bar E_B| - \sum_{j\notin B} |\bar E_{\{j\}}|}  = \one_{B= \{1\etc n\}} + O(N^{-1}).
	\eeqy
	
Moreover, by  \cite[Lemma 1.1]{GUI}  $|\sigma| - c_\sigma -|\bar E_\sigma|$ is the number of cycles of $\bar T_\sigma$, the graph obtained from $T_\sigma$ by forgetting the multiplicity and the orientation of its edges. Hence, 
	\beqy\label{eq:Contr3}
		 N^{|\sigma| - c_\sigma -|\bar E_\sigma|}= \one_{\bar T_\sigma \textrm{ is a forest}} + O(N^{-1}).
	\eeqy

 By \eqref{eq:Contr1}, \eqref{eq:Contr2} and \eqref{eq:Contr3}, if we denote by $\delta_N(\sigma) = \delta_N(\{1\etc n\})$ we get
	\beqy
		\lefteqn{\esp \Big[ Z_N(T_1) \dots Z_N(T_n) \Big]}\\
		 & = & \sum_{\pi \in \mathcal P_2(n)} \ \sum_{\substack{\sigma \in \mathcal P_\sharp(V_1\etc V_n) \\ \textrm{s.t. } \bar \sigma = \pi}} \one_{\bar T_\sigma \textrm{ is a forest of } \frac n 2 \textrm{ trees }} \  \delta_N(\sigma) + O( N^{-1})\nonumber\\
		 & = &\sum_{\pi \in \mathcal P_2(n)} \ \prod_{\{i,j\} \in \pi} \sum_{ \sigma \in \mathcal P_\sharp(V_i, V_j) } \one_{\bar T_\sigma \textrm{ is a tree }} \  \delta_N(\sigma) + O( N^{-1}),\la{14149}
	\eeqy
	
	\noindent where we have used the independence of the entries of $A$ to split $\delta_N$. The case $n=2$ gives
	\beq
		\esp \Big[ Z_N(T_1)   Z_N(T_2) \Big] = \underbrace{\sum_{\sigma \in \mathcal P_\sharp(V_1,V_2)} \one_{ \bar T_\sigma \textrm{ is a tree}} \ \delta(\sigma)}_{M^{(2)}(T_1,T_2)} + o(1),
	\eeq
where $\delta(\sigma) = \Nlim \delta_N(\sigma)$, which exists since $\bar T_\sigma $ is a tree. Indeed, in the definition \eqref{eq:SecondExpanMoments} of $\delta(\sigma)$, we have $p_k = q_{k,\ell}=0$ for any $k \neq \ell$. Moreover, we obtain that $\delta(\sigma) = \tau^0[T_\sigma]$ defined in Proposition \ref{Prop:CVDistrTraff}, and the sum over $\sigma$ on $\mathcal P_\sharp(V_1,V_2)$ can be replaced by a sum over graphs $T$ obtained by identifying certain adjacent vertices of $T_1$ with adjacent edges $T_2$, since $\tau^0[T]$ is zero if $T$ is not a fat tree. We then obtain as expected the limiting covariance
		$$M^{(2)}(T_1,T_2) = \sum_{ T \in \mathcal P_\sharp(T_1,T_2)} \tau^0[ T].$$

 The general case $n\geq 3$ in \eqref{14149} gives the Wick formula
	\beq
		 \esp \Big[ Z_N(T_1) \dots Z_N(T_n) \Big] = \sum_{\pi \in \mathcal P_2(n)} \ \prod_{\{i,j\} \in \pi} M^{(2)}(T_i,T_j)+ o(1),
	\eeq
 which characterizes the Gaussian distribution.
\end{proof}

Remark that  up to \eqref{14149} the errors terms are of order $O(N^{-1})$, and so if $N \esp [a_{11}^{2k}] = C_k + o(\sqrt N^{-1})$ for any $k\geq 1$, then $\delta_N(B) = \delta(B) + o(\sqrt N^{-1})$ so that we find that
$$\esp\Big[ \frac{1}{N}\Tr A^k\Big]=\lim_{N\ra\infty}\esp\Big[ \frac{1}{N}\Tr A^k\Big]+o( N^{-1/2})$$
and therefore we  obtain the same CLT if we recenter with  the limit or the expectation, as noticed in the introduction.


\section{CLT for  Stieltjes transform and the method of martingales}\la{sec:prCLTSt10713}

Let $A=[a_{i,j}]$ be an $N\ti N$ matrix satisfying Hypothesis \re{Hyp:Model}.  
First of all, by the resolvant formula, we know that for any $z\in \C\bck \R$, there is a constant $C$ \st for any pair $X, Y$ of $N\ti N$ symmetric matrices,  $$|\Tr ((z-X)^{-1}-(z-Y)^{-1})|\le C\op{rank}(X-Y).$$ Hence one can suppose that in Hypothesis \re{Hyp:Model}, $b$ is centered.

For any $z$ in $\mbb C\setminus \mbb R$, recall that $G(z) = (z - A)^{-1}$. To prove Theorem \ref{TCLHTRM}, we show that any linear combination of the random variables
		$$Z_N(z):=\frac 1 {\sqrt N} \Tr G(z) - \esp\Big[ \frac 1 {\sqrt N} \Tr G(z)  \Big], \ z \in \mbb C \setminus \mbb R,$$
and their complex transposes converges in distribution to a complex Gaussian variable. Since $\overline{G(z)} = G(\bar z)$,  it is enough to fix a linear combination
	 \be\la{309129:33}
	 	M(N):=\sum_{i=1}^p\lam_i  Z_N(z_i),
	\ee
for some fixed $p\ge 1$ and $\lam_1, \ld, \lam_p\in \C$, $z_1,\ld, z_p\in \C\bck\R$, and prove that $M(N)$ is asymptotically Gaussian with the adequate covariance.
 
\subsection{The method of martingales differences and reduction to the case $p=1$}\label{sec:Step1Levy}
For all $N$, we have  
		$$M(N)-\E[M(N)]=\sum_{k=1}^N X_k\qquad\trm{ with }\quad	X_k:=(\E_k-\E_{k-1})\big[M(N)\big],$$ 
where $\E_k$ denotes the conditional  expectation with respect to the $\si$-algebra generated by the $k\ti k$ upper left corner of $A$. In view of Theorem \re{thconvmart} of the appendix about fluctuations of martingales, 
it is enough  to prove that  we have   
	\be\la{summodsADTC} 
		\sum_{k=1}^N\E_{k-1}[X_k^2]\Ninf \sum_{i,j=1}^p\lam_i\lam_j C(z_i,{z_j}),
	\ee 
	\be\la{summodsADTCbis} 
		\sum_{k=1}^N\E_{k-1}\big[|X_k|^2\big]\Ninf \sum_{i,j=1}^p\lam_i\ovl{\lam_j} C(z_i,\ovl{z_j})
	\ee 
 and that for each $\eps>0$,  \be\la{epslimit}\sum_{k=1}^N \E[|X_k|^2\one_{|X_k|\ge \eps}]\Ninf 0. \ee

Notice first that \eqre{summodsADTC}  implies \eqre{summodsADTCbis}.
 Let us now prove  \eqre{epslimit}. The proof of \eqre{summodsADTC}  will then be the main difficulty of the proof of Theorem \re{TCLHTRM}.

\noindent{\it Proof of  \eqre{epslimit}.} Let $A_k$ be the symmetric matrix with size $N-1$ obtained by removing the $k$-th row and the $k$-th column of $A$ and set $G_k(z):=\ff{z-A_k}$. 
Note that $\E_k G_k(z)=\E_{k-1}G_k(z)$, so that we can write
$$X_k=\sum_{i=1}^p\lam_i\ti (\E_k-\E_{k-1})[\ff{\sqrt{N}}(\Tr G(z_i)- \Tr G_k(z_i))].$$ 
 Hence by \eqre{307122} of Lemma \re{lat} in the appendix, there is $C$ \st  
 for all $N$ and all $k$, $$|X_k|\le  \f{C \max_i|\lam_i||\Im z_i|^{-1}}{\sqrt{N}}.$$ Thus for $N$ large enough, we have that for all $k$, $|X_k|^2\one_{|X_k|\ge \eps}=0$ and  \eqre{epslimit} is proved.\hfill$\square$\\ \\

 We now pass to the main part of the proof of Theorem \re{TCLHTRM}, namely the proof of \eqre{summodsADTC}. 
 It is divided into several steps.\\ \\ 
 
 We can get rid of the linear combination in \eqre{309129:33} and assume $p=1$.
 As $\ovl{\Tr G(z_i)}=\Tr G(\ovl{z_i})$, both 
 $X_k^2$ and $|X_k|^2$ are linear combinations of terms of the form $$(\E_k-\E_{k-1})[\ff{\sqrt{N}}\Tr G(z)]\ti (\E_k-\E_{k-1})[\ff{\sqrt{N}}\Tr G(z')],$$ with $z,z'\in \C\bck\R$. 
 As a consequence, we shall only  fix $z\in \C\bck\R$ and prove that for 
	$$
		Y_k(z):=(\E_k-\E_{k-1})[\ff{\sqrt{N}}\Tr G(z)],
	$$
we have   the convergence in probability for any $z,z'\in  \C\bck\R$
	\be\la{summods} 
		C_N:=\sum_{k=1}^N\E_{k-1}[Y_k(z)Y_k(z')]\Ninf  C(z,z').
	\ee 

 First, for $G_k$ as introduced in the proof of  \eqre{epslimit} above, as  $\E_k G_k(z)=\E_{k-1}G_k(z)$ again, we have $Y_k(z)=(\E_k-\E_{k-1})[\ff{\sqrt{N}}(\Tr G(z)- \Tr G_k(z))]$. Hence by   Lemma \re{lat},   we can write
 \be\la{187121}Y_k(z) =\ff{\sqrt{N}}(\E_k-\E_{k-1})\f{1+ \ba_k^*G_k(z)^2\ba_k}{z- a_{kk} - \ba_k^*G_k(z)\ba_k },\ee where $\ba_k$ is the $k$-th column of $A$ where the $k$-th entry has been removed.
To prove \eqref{summods}, we shall first show that we can get rid of the off diagonal terms    $\sum_{j\neq\ell} \ba_{k}(j)\ba_{k}(\ell) G_k(z)_{j\ell}$  and  $\sum_{j\neq\ell} \ba_{k}(j) \ba_{k}(\ell) (G_k(z)^2)_{j\ell}$   in the above expression.

\subsection{Removing the off-diagonal terms}\label{sec:Step2Levy}

In this section, we prove that we can 
   replace ${Y}_k$ in \eqref{summods} by \be\la{187125}  \tY_k(z):=\ff{\sqrt{N}}(\E_k-\E_{k-1})\f{1+\sum_{j=1}^{N-1}\ba_k(j)^2(G_k(z)^2)_{jj}}{z-\sum_{j=1}^{N-1}\ba_k(j)^2G_k(z)_{jj} }.\ee  
 Note first that by  Equation \eqre{297125} of  Lemma \re{lem267121} in the appendix, we have the deterministic bound,  for all $z\in \mathbb C\backslash\mathbb R$,
 $$\lf|\f{1+\sum_j\ba_k(j)^2(G_k(z)^2)_{jj}}{z-\sum_j\ba_k(j)^2G_k(z)_{jj} }\ri|\le 2|\Im z|^{-1}$$
 hence  we deduce
  \be\la{297129h42}| \tilde{Y}_k| \le\f{4|\Im z|^{-1}}{\sqrt{N}}.\ee

\beg{lem}\la{20812} Let  $\tY_k(z')$ be defined in   the same way as $\tY_k(z)$ in \eqre{187125}, replacing $z$ by $z'$. Set also      
\be\la{187125b}\tilde{C}_N:=\sum_{k=1}^N\E_{k-1}[\tilde{Y}_k(z)\tilde{Y}_k(z')]. \ee Then for $C_N$ defined as in  \eqre{summods},  as $N\to\infty$,  we have the convergence  \be\la{19712_12h} {C}_N - \tilde{C}_N\cvLone 0.\ee
\en{lem}

\bpr 
We have $${C}_N - \tilde{C}_N=\sum_{k=1}^N\E_{k-1}[{Y}_k(z){Y}_k(z')-\tilde{Y}_k(z)\tilde{Y}_k(z')]=\int_{u=0}^1h_N(u)\ud u,$$
with $h_N(u):=N\E_{k-1}[{Y}_k(z){Y}_k(z')-\tilde{Y}_k(z)\tilde{Y}_k(z')]$ for $k=\lceil Nu\rceil$.
As we already saw, by Lemmas \re{lat} and  \re{lem267121} of the appendix, there is a constant $C$ (independent of $k$ and $N$) \st $  |{Y}_k|$, $|{Y}_k'|$, $|\tilde{Y}_k|$ and $|\tY_k(z')|$ are all  bounded above by $C/\sqrt{N}$, so that $|h_N(u)|\le 2C$. Hence by dominated convergence, it suffices to prove that for any fix $u$, $\|h_N(u)\|_{L^1}\lto 0$, which is equivalent to $h_N(u)\lto 0$ in probability. This is a direct consequence of Lemma \re{vanishing_nondiagterms_10713} of the appendix (which can be applied here because we explained, at the beginning of Section \re{sec:prCLTSt10713}, that one can suppose the random variable $b$ of Hypothesis \re{Hyp:Model} to be centered) and of the fact that the function of two complex variables $\vfi(x_1, x_2):=\f{1+x_1}{z-x_2}$ has a uniformly bounded gradient on the set $\{(x_1,x_2)\in \C^2\ste \Im x_2\le 0\}$.\epr

It remains to prove that  the sequence $\tilde{C}_N$ introduced at \eqre{187125b} converges in \pro as $N$ goes to infinity. Note that we have \be\la{148129h19}\tilde{C}_N= \int_{u=0}^1N\E_{k-1}[\tilde{Y}_{k(u)}(z)\tilde{Y}_{k(u)}(z')]\ud u \qquad\trm{ with $k(u):={\lceil Nu\rceil}$}. 
\ee
By \eqre{297129h42}, the integrand is uniformly bounded by $4|\Im z|^{-1}$. Hence, by dominated convergence, it is enough to prove that for any fixed 
$u\in (0,1)$, as $k,N$ tend to infinity in such a way that $k/N\lto u$, we have the convergence in probability of
	$N\E_{k-1}[\tilde{Y}_{k}(z)\tilde{Y}_{k}(z')].$

Now, set for $z,z'\in \C\backslash \R$,
 \be\la{2071215}f_k:=\f{1+\sum_j\ba_k(j)^2(G_k(z)^2)_{jj}}{z-\sum_j\ba_k(j)^2G_k(z)_{jj} },\quad f_k':=\f{1+\sum_j\ba_k(j)^2(G_k(z')^2)_{jj}}{z-\sum_j\ba_k(j)^2G_k(z')_{jj} }.\ee 
We have $\sqrt{N}\tY_k(z)=(\E_k-\E_{k-1})f_k$, so for  $z,z'\in \C\backslash \R$,
  $$N\E_{k-1}(\tY_k(z)\tY_k(z'))=\E_{k-1}[\E_kf_k\E_kf_k']-\E_{k-1}f_k\E_{k-1}f_k'.$$ Let us denote by $\E_{\ba_k}$ the expectation with respect to the randomness of the $k$-th column of $A$ (\ie the conditional expectation with respect to the $\si$-algebra generated by the $a_{ij}$'s \st $k\notin \{i,j\}$). Note that $\E_{k-1}=\E_{\ba_k}\circ\E_k=\E_{k}\circ \E_{\ba_k}$,  hence 
   \begin{equation}\la{1481212h51bis} 
 N\E_{k-1}(\tY_k(z)\tY_k(z'))=\E_{\ba_k}[\E_kf_k\E_kf_k']-\E_{k}\E_{\ba_k}f_k\ti \E_{k}\E_{\ba_k}f_k'.
\end{equation} Now, we introduce (on a possibly  enlarged \pro space where the conditional expectations $\E_{k-1}, \E_k,\E_{\ba_k}$ keep their definitions as the conditional expectations with respect to the same $\si$-algebras as above) an  $N\ti N$  random symmetric matrix 
	\be\la{88123}
		A'=[a_{ij}']_{1\le i,j\le N}
	\ee
 such that: \bgt\ite  the  $a_{ij}'$'s \st $i>k$ or  $j>k$ are i.i.d. copies of $a_{11}$ (modulo the fact that $A'$ is symmetric),     independent of $A$,
\ite for all other pairs $(i,j)$, 
$a_{ij}'=a_{ij}$,
\ent then we have $$\E_kf_k\E_kf_k'=\E_k(f_k\ti f_k''),$$ where $f_k''$ is  defined out of $A'$   in the same way as $f_k'$ is defined out of $A$ in \eqre{2071215} (note that the $k$-th column is the same in $A$ and in $A'$).
It follows that 
\beqy\nonumber N\E_{k-1}(\tY_k(z)\tY_k(z'))&=&\E_{\ba_k}[\E_k(f_k\ti f_k'')]-\E_{k}\E_{\ba_k}f_k\ti\E_{k}\E_{\ba_k}f_k'\\ \la{247121}
&=&\E_{k}[\E_{\ba_k}(f_k\ti f_k'')]-\E_{k}\E_{\ba_k}f_k\ti\E_{k}\E_{\ba_k}f_k'.\eeqy

 We shall in the sequel prove that as $N$ tends to infinity, regardless to the value of $k$, we have the almost sure convergences 
 \be\la{1481214h45}\E_{\ba_k}f_{k}\;\lto\; L(z)\qquad\qquad \E_{\ba_k}f_k'\;\lto\; L(z')
 \ee
 and for any $u\in (0,1)$, as $N,k\lto \infty$ so  that $k/N\lto u$, we have 
  \be\la{1481214h45bis}\E_{\ba_k}(f_{k}\ti f_{k}'')\;\lto\; \Psi^u({z,z'}) \, a.s,\ee
where $L(z,z') = \int_0^1 \Psi^u(z,z')du$.  The convergences of \eqre{1481214h45} and \eqre{1481214h45bis} are based on an abstract convergence result stated in next section, where we use the convergence of generalized moments of Proposition \ref{Prop:CVDistrTraff}. They are stated in Lemmas \ref{Lem:CVdeL} and \ref{Lem:CVdePsi} respectively.
  
 Note that by Lemma \re{lem267121},   \be\la{148129h42jb}| f_k| \le 4|\Im z|^{-1} ,\ee so once \eqre{1481214h45} and \eqre{1481214h45bis} proved, we will have the convergence in $L^2$ 
 	$$\E_{k}[\E_{\ba_k}(f_{k(u)}\ti f_{k(u)}'')]-\E_{k}\E_{\ba_k}f_k\ti\E_{k}\E_{\ba_k}f_k'\;\lto\;\Psi^u({z,z'})-L(z)L(z').$$

Thus by \eqre{247121}, we will have proved the convergence of \eqre{1481212h51bis}, hence completing the proof of the theorem.

\subsection{An abstract convergence result}\label{sec:Step3Levy}

Remember that $A_k=(a_{ij})$ is the square matrix of size $N-1$ obtained by removing the $k$-th row and the $k$-th column of $A$, that $G_k(z)  = (z - A_k)^{-1}$ and that $A'_k$ is a copy of $A_k$ where the entries $(i,j)$ for $i$ or $j\geq k$ are independent of $A_k$ and the other are those of $A_k$. We denote $G_k'(z') = (z' - A'_k)^{-1}$.

\beg{lem}\la{Lem:CVparlesTraffics} Under Hypothesis \ref{Hyp:Model} and condition \eqref{CondMom}, as $N$ goes to infinity and $\frac k N$ tends to $u$ in $(0,1)$, the random \pro measure on $\C^2$ 
		\be\la{118121}
			\nu^{k,N}_{z,z'}:=\frac 1 {N-1} \sum_{j=1}^{N-1} \del_{ \{ G_k(z)_{jj} , G'_k(z')_{jj}  \}}, \ z,z' \in \mbb C \setminus \mbb R.
		\ee
converges weakly almost surely 
 to a deterministic probability measure $\nu_{z,z'}^u$ on $\mbb C^2$.
\en{lem}

\bpr By e.g. Theorem C.8 of \cite{alice-greg-ofer}, it is enough  to prove that for any bounded and Lipschitz function $f$, $\nu^{k,N}_{z,z'}(f)$ converges   almost surely  to $\nu_{z,z'}^u(f)$. Moreover,  adapting the proof of Lemma \ref{le:concres}, one can easily see that for any such $f$, 
		$$\nu^{k,N}_{z,z'}(f)-\E [\nu^{k,N}_{z,z'}(f)]$$
converges almost surely  to zero. The only modification of the proof is to complete the resolvent identity by noticing that 
	$$(z -A_k)^{-2} - (z-A_k^B)^{-2} = (z-A_k^B)^{-2} \big( z(A_k-A_k^B) - (A_k - A_k^B)^2 - 2 A_k^B(A_k - A_k^B) \big) (z-A_k)^{-2},$$
which gives that this matrix has rank bounded by $3 \times \op{rank}(A-B)$. Hence it is enough  to prove that the deterministic sequence of measures $\E [\nu^{k,N}_{z,z'}(\;\cdot\;)]$ converges weakly (the measure $\E [\nu^{k,N}_{z,z'}(\;\cdot\;)]$ is defined by $\E [\nu^{k,N}_{z,z'}(\;\cdot\;)](f)=\E [\nu^{k,N}_{z,z'}(f)]$ for any bounded Borel function $f$ : one can easily verify that this is actually a \pro measure). Moreover, the measure $\E [\nu^{k,N}_{z,z'}(\;\cdot\;)]$ always belongs to the set of \pro measures supported by    the compact subset $B(0, |\Im z|^{-1})^2$. The set of such \pro measures is compact for the topology of weak convergence. Hence it suffices to prove that it admits at most one accumulation point.

Let us first explain how the $a=b_\eps+c_\eps$ decomposition of Hypothesis \re{Hyp:Model} allows to reduce the problem to the case of Wigner matrices with exploding moments. For any $\eps>0$, the matrix $A^\eps:=[b_{\eps,ij}-\E b_{\eps,ij}]_{i,j=1}^N$ is a Wigner matrix with exploding moments by Hypothesis \re{Hyp:Model}. 

{\bf Claim : } Almost surely, $\ds\limsup_{N\to\infty} \f{\op{rank}(A-A^\eps)}{N} \le \eps$. 
\\
Indeed, the rank of $A-A^\eps$ is at most $1+\sum_{i=1}^N\one_{\exists j\le i, \, c_{\eps,ij}\ne 0}.$ Hence, by  independence
and  concentration inequalities such as Azuma-Hoeffding inequality,  
it suffices to prove that $$\limsup_{N\to\infty}\ff{N}\sum_{i=1}^N\p(\exists j\le i, \, c_{\eps,ij}\ne 0)\le \f{\eps}{2}.$$
But we have, for $p:=\p(c_\eps\ne 0)$,  $$\ff{N}\sum_{i=1}^N\p(\exists j\le i, \, c_{\eps,ij}\ne 0)=\ff{N}\sum_{i=1}^N1-(1-p)^i=1-(1-p)\f{1-(1-p)^N}{pN}$$
As $pN$ is small, $(1-p)^N=1-Np+O(Np^2)$,   $$\ff{N}\sum_{i=1}^N\p(\exists j\le i, \, c_{\eps,ij}\ne 0) =  O(Np)$$ and the Claim is proved, by Hypothesis \eqre{limceps}.

We can adapt arguments of \cite{BAGheavytails} 
(see also Equation (91) of \cite{charles_alice}) to see that  it is enough to prove the weak convergence of $\E \big [\nu^{k,N}_{z,z'}(\;\cdot\;) \big]$ when $A$ is a Wigner matrix with exploding moments. Since these measures are uniformly compactly supported, it is sufficient to prove the convergence in moments.

We consider a polynomial $P=x_1^{n_1} \overline{ x_1}^{m_1}x_2^{n_2}\overline{ x_2}^{m_2}$ and remark that
	\eq
		\E \big [\nu^{k,N}_{z,z'}(P) \big] & = & \esp \bigg[ \frac 1 {N-1} \sum_{j=1}^{N-1} \big( G_k(z)_{jj} \big)^{n_1}\big( G_k(\bar z)_{jj} \big)^{m_1} \big( {G'_k}(z')_{jj} \big)^{n_2} \big( {G'_k}(\bar z')_{jj} \big)^{m_2} \bigg]\\
		& = & \esp \bigg[ \frac 1 {N-1} \Tr \Big[ G_k(z)^{\circ n_1} \circ G_k(\bar z)^{\circ m_1} \circ G'_k(z')^{\circ n_2}\circ G'_k(\bar z')^{\circ m_2}   \Big] \bigg],
	\qe
where $\circ$ denotes the Hadamard (entry-wise) product of matrices and 
		$$ M^{\circ n}:= \underbrace{ M \circ \dots \circ M}_{n}.$$
We set $\ell_1 = n_1 + m_1$ and $\ell_2 = n_2+m_2$. Let $(Y_i,Y'_j)_{i=1\etc \ell_1, j=1\etc \ell_2}$ be a family of random variables such that for any $p_i, q_j\geq 0$,
	\begin{equation} \label{eq:HadamardTraceFormula}
		\esp \Big[ \prod_{i=1}^{\ell_1} Y_i^{p_i}  \prod_{j=1}^{\ell_2} {Y'_j}^{q_j} \Big]
		 =  \esp \bigg[ \frac 1 {N-1} \Tr \Big[ A_k^{p_1} \circ \dots \circ A_k^{p_{\ell_1} } \circ {A_k'}^{q_1} \circ \dots \circ {A_k'}^{q_{\ell_2}} \Big] \bigg]. 
	\end{equation}
Such a family exists by Proposition \ref{Prop:Existence}. By \cite[Proposition 3.10]{MAL122}, the couple of random matrices $(A,A')$ satisfies the so-called convergence in distribution of traffics,  so that the RHS converges. For the reader's convenience, we give the limiting value of \eqref{eq:HadamardTraceFormula}, even if we do not use it later. It is obtained by applying the rule of the so-called traffic freeness in \cite{MAL122}.
	\eqa \label{Eq:CalculFreeness}
		\esp \Big[ \prod_{i=1}^{\ell_1} Y_i^{p_i}  \prod_{j=1}^{\ell_2} {Y'_j}^{q_j} \Big] \limN \sum_{ \pi \in \mathcal P(V_T) } \tau^0[T^\pi] \alpha(T^\pi),
	\qea
where
\begin{enumerate}
	\item[1.] we have considered $T$ the graph whose edges are labelled by indeterminates $a$ and $a'$, obtained by
		\begin{itemize}
			\item considering the disjoint union of the graphs with vertices $1 \etc p_i$ and edges $\{1,2\} \etc \{p_i-1,p_i\}, \{p_i,1\}$ labelled $a$, $i=1\etc \ell_1$,
			\item considering the disjoint union of the graphs with vertices $1 \etc q_j$ and edges $\{1,2\} \etc \{q_j-1,q_j\}, \{q_j,1\}$ labelled $a'$, $j=1\etc \ell_2$,
			\item identifying the vertex $1$ of each of these graphs (we get a connected graph, bouquet of cycles),
		\end{itemize}
	\item[2.] we have denoted by $V_T$  the set of vertices of $T$, $\mathcal P(V_T)$ is the set of partitions of $V_T$ and $T^\pi$ denotes the graph obtained by identifying the vertices of $T$ that belong to a same block of $\pi$,
	\item[3.] the quantity $\tau^0[T^\pi]$ is as in Proposition \ref{Prop:CVDistrTraff},
	\item[4.] we have set
			$$\alpha(T^\pi) = \sum_{V_\pi = V_1 \sqcup V_2} \one_{ \left\{  \substack{
			    \textrm{    the edges which are not linking }         \\ 
			     \textrm{adjacent vertices of } V_1 \textrm{ have the same label}       }\right \}} u^{|V_1|} (1-u)^{ |V_2|},$$
		 where the sum is over all partitions of the set $V_\pi$ of vertices of $T^\pi$ and $|E_i|$ is the number of edges between adjacent vertices of $V_i$, $i=1,2$.
\end{enumerate}

Formula \eqref{Eq:CalculFreeness} could also be derived by the same techniques than those developed in Section \ref{Sec:CLTmoments}. The random variables $Y_i$ and $Y'_j$ are distributed according to the limiting eigenvalues distribution of $A$. Recall that we assume that the sequence $(C_k)_{k\geq 2}$ defined in \eqref{1971214h} satisfies $C_k \leq C^k$ for a constant $C>0$. Then, following the proof of \cite[Proposition 10]{ZAK}, the exponential power series of the limiting eigenvalues distribution of $A$ has a positive radius of convergence. So, by a generalization of \cite[Theorem 30.1]{Billingsley} to the multi-dimensional case, we get that the distribution of $(Y_i,Y'_j)$ is characterized by its moments. Then, we get that $(Y_i,Y'_j)$ converges weakly to a family of random variables $(y_i,y_j')$. We set $f_z(y) = (z-y)^{-1}$. We then obtain the convergence
	\eq
		\E \big [\nu^{k,N}_{z,z'}(P) \big]  & = & \esp \bigg[ \prod_{i=1}^{n_1} f_z(Y_i)   \prod_{j=1}^{n_2} f_{z'}(Y'_j) \bigg] \limN   \esp \bigg[ \prod_{i=1}^{n_1} f_z(y_i)   \prod_{j=1}^{n_2} f_{z'}(y'_j)  \bigg]
	\qe

\noindent Hence  $\E \big [\nu^{k,N}_{z,z'}(\;\cdot\;) \big]$ converges weakly.
\end{proof} 

This convergence could also be proven without Proposition \ref{Prop:Existence} but with appropriate bounds on the growth of moments.
We have the following Corollary.
\begin{cor}\label{corconv1}
\begin{enumerate}
\item 
For $z\in\mathbb C\backslash\mathbb R$, $t$ so that $\Im z t\ge 0$,
$$\rho^N_z(t):=\frac{1}{N-1}\sum_{i=1}^{N-1}  \Phi( tG_k(z)_{ii})\limN \rho_z(t):=\int \Phi(tx)\nu^0_{z,z'}(dx,dx')\qquad a.s.$$
\item For $z,z'\in\mathbb C\backslash\mathbb R$, $t, t'$ so that $\Lambda= \{ t \Im z \ge 0, t'\Im z'\ge 0, |t|+|t'| >0\}$
and $k/N$ going to $u$,
$$
\rho^{N,k}_{z,z'}(t,t'):=\frac{1}{N-1}\sum_{i=1}^{N-1}  \Phi( tG_k(z)_{ii}+t'G'_k(z')_{ii})$$
converges almost surely towards 
	\beqy \label{Eq:DefRhozz'}
		\rho^u_{z,z'}(t, t')=\int \Phi(tx+t'x')\nu^u_{z,z'}(dx,dx')\,.
	\eeqy
\item For any $(t,t')\in\mathbb R$, the functions $\rho^{N, k}(t,t')$ are analytic on $\Lambda=\{z,z': t  \Im z >0,  t'\Im z'>0\}$ and uniformly bounded on compact subsets of $\bar\Lambda$. They have a non positive real part.
Their limits are also analytic on $\Lambda$ and have a non positive real part. 
\item For all $z\in\mathbb C^+$, 
$$\lim_{N\ra\infty}\frac{1}{N}\Tr((z-A)^{-1}) = i\int_0^\infty e^{itz +\rho_z(t)} dt\qquad a.s.$$
\end{enumerate}
\end{cor}
\bpr Let us first notice that  $\rho^N_z(t)=\rho^{N,k}_{z,z'}(t,0)$ for all $k$ so we only need to focus on
$\rho^{N,k}_{z,z'}(t,t')$. The point wise convergence of the function $\rho^{N,k}$ is a direct consequence of the  continuity of $\Phi$ (recall we assumed that $\Phi$ extends continuously to
the real line), of the boundedness of $ tG_k(z)_{ii}+t'G'_k(z')_{ii}$  (by $t/\Im z+ t'/\Im z'$) and Lemma \ref{Lem:CVparlesTraffics}.
To show analyticity, note that 
for all $j\in \{1,\ldots,N\}$,
$G(z)_{jj}$ is an analytic function on $\C^+$, taking its values in $\C^-$ (and vice versa)  almost surely by \eqre{297122}. Hence, on $\Lambda$, $ tG_k(z)_{ii}+t'G'_k(z')_{ii}$ is an analytic function
with values in $\C^-$. 
Therefore, as $\Phi$ is analytic on $\C^{-}$,  $\rho^{N,k}$ is an analytic function  on $\Lambda$ almost surely. Moreover, as $\Phi$ extends continuously to the real line, it is uniformly 
bounded on compact subsets of $\ovl{\C^-}$ and hence $\rho^{N,k}$ is uniformly bounded on compact subsets of $\Lambda$. This implies by Montel's theorem that 
the limit $\rho^u$ of $\rho^{N,k}$ is also analytic. Finally, $\rho^{N,k}$ as non positive real part as the image of $\C^-$ by $\Phi$. Indeed, as $\Re(-i\lam| a_{11}|^2)\le 0$,  we have
 \be\la{AD5111200:28}|\phi_N(\lam)| =|\E(e^{-i\lam |a_{11}|^2})|\le 1\quad\Rightarrow \Re\Phi(\lambda)=\lim_{N\ra\infty}
N\log |\phi_N(\lambda)|\le 0.\ee  
For the last point, first note that  by Lemma \re{le:concspec}, it is enough to prove the result for $\E[N^{-1}\Tr G(z)]$ instead of  $N^{-1}\Tr G(z)$. Second, by exchangeability, $\E[N^{-1}\Tr G(z)]=\E[G(z)_{11}]$. Remind that $\ba_k$ is the $k$-th column of $A$ (or $A'$) where the $k$-th entry has been removed. Using Schur complement formula and getting rid of the off diagonal terms (by arguments similar to those of Section \ref{sec:Step2Levy}, \ie essentially Lemma \re{vanishing_nondiagterms_10713}), we have for any $z\in\bC^+$
	\beq
		\E[ G(z)_{11}] &=& \mathbb E\Big[ \frac{1}{z-\sum_{j=1}^{N-1} |\ba_{1}(j)|^2 G_1(z)_{jj} +\varepsilon_N}\Big]\\
		&= &  \mathbb E\Big[ \frac{1}{z-\sum_{j=1}^{N-1} \ba_{1}(j)^2 G_1(z)_{jj} }\Big] +o(1).
	\eeq
Remember that we have for $\lambda \in \C \bck\R$,  
	\be\la{rep}
		\ff{\lam}=-i\int_{0}^{\op{sgn}_\lam\infty}e^{it\lam}\ud t,
	\ee
and that $\phi_N(\lambda) = \E \big[ \exp( -i \lambda |a_{11}|^2) \big]$, which gives
	\beq
		\mathbb E[ G(z)_{11}] &=& i \int_0^\infty \mathbb E[ e^{i\lambda(z-\sum_{j=1}^{N-1} \ba_{1}(j)^2 G_1(z)_{jj})}]\ud\lambda+o(1)\\
&=&i\int_0^\infty e^{i\lambda z}\mathbb E\Big[ \prod_{j=1}^{N-1} \phi_N(\lambda G_1(z)_{jj}) \Big]\ud\lambda+o(1).
	\eeq
We used the exponential decay to switch the integral and the expectation. This also allows to truncate the integral: for any $ M\geq 0$,
		$$\int_0^\infty e^{i\lambda z}\mathbb E\Big[ \prod_{j=1}^{N-1} \phi_N(\lambda G_1(z)_{jj}) \Big]\ud\lambda = \int_0^M e^{i\lambda z}\mathbb E\Big[ \prod_{j=1}^{N-1} \phi_N(\lambda G_1(z)_{jj}) \Big]\ud\lambda +\epsilon_{m,M,z,N},$$
where $\epsilon_{M,z,N}$ goes to zero as $M$ tends to infinity, uniformly on $N$ and on the randomness.
Remember  that by assumption \eqref{2071216h33}, we have 
	\beqy\label{eq:CvPhi}
		 N\big(\phi_N(\lambda) -1\big) \limN \Phi(\lambda), \ \forall \lambda \in \ovl{\C^-}
	\eeqy
	where the convergence is uniform on compact subsets of $\ovl{\C^-}$. Hence, since for $|\Im z|\ge \delta>0$, $|t|\le M$,
	$ t G_k(z)_{jj}$ belongs to the compact set $ \{\lambda\in\ovl{\C^-}\ste |\lambda|\le M/\delta\}$, we have 
	$$\prod_{j=1}^{N-1}\phi_N(tG_k(z)_{jj}) = \exp\Big(\frac{1}{N} \sum_{j=1}^{N-1}\Phi( t G_k(z)_{jj}) \Big) + \varepsilon^{(1)}_{t,z,N} =\exp(\rho^N_z(t) ) +  \varepsilon^{(1)}_{t,z,N} $$
where $\varepsilon^{(1)}_{t,z,N}$ converges almost surely to zero as $N$ goes to infinity. 
By Corollary \ref{corconv1}, 
	$$\exp\Big(\frac{1}{N} \sum_{j=1}^{N-1}\Phi( t G_k(z)_{jj}) \Big)  =\exp\big( \rho_z(t) \big) + \varepsilon^{(2)}_{t,z,N},$$
where $ \varepsilon^{(2)}_{t,z,N}$ converges to zero almost surely. Hence, we deduce  the almost sure convergence 
	\beqy\label{eq:CVversRho}
		   e^{itz} \prod_{j=1}^{N-1}\phi_N(tG_k(z)_{jj})  \limN e^{itz + \rho_z(t)}.
	\eeqy
	As $\rho_z^N$ has non positive real part, we  conclude by dominated convergence theorem and by getting rid of the truncation of the integral.
\epr

\subsection{Computation and convergence  of $\E_{\ba_k} f_k$ and $\E_{\ba_k}f_k'$}\label{sec:Step4Levy}

\begin{lem}\label{Lem:CVdeL} Almost surely, we have the convergence
\be
	\E_{\ba_k} f_k \limN L(z):=
 -  \int_0^{\op{sgn}_z \infty} \frac {1} t \partial_z e^{  itz+\rho_z(t)}\ud t\,.
 	\ee
where $\op{sgn}_z:=\op{sgn}(\Im z )$ and
	$ \rho_z(t)$ is defined in Corollary \ref{corconv1}.
	\end{lem}

\begin{proof}Remember that for any $z\in \C\bck\R$, 
	$$f_k=\f{1+\sum_j\ba_k(j)^2(G_k(z)^2)_{jj}}{z-\sum_j\ba_k(j)^2G_k(z)_{jj} }:= \frac{ \partial_z \lambda_N(z)}{\lambda_N(z)}.$$
Hence, by \eqref{rep}, since by \eqref{297122} the sign of the imaginary part of $\lambda_N(z)$ is $\op{sgn}_z$, the random variable $f_k$ can be written 
	\beqy\la{88121}
		f_k & = & -i \int_0^{\op{sgn}_z\infty}\partial_z \lambda_N(z) e^{it\lambda_N(z)}\ud t.
		\\
			& =&    -i \int_{\op{sgn}_z  m}^{\op{sgn}_z\infty}\partial_z \lambda_N(z) e^{it\lambda_N(z)} \ud t    -i \int_{0}^{\op{sgn}_zm} \partial_z \lambda_N(z) e^{it\lambda_N(z)} \ud t \nonumber \\
			& := &  \tilde f_{k,m} + \eta_m(z),\nonumber
	\eeqy
where $m>0$ and $\eta_m(z) = \frac{ \partial_z \lambda_N(z)}{ \lambda_N(z)} \big( 1- e^{i \op{sgn}_z m \lambda_N(z)} \big)$.

We next show that for all $\varepsilon>0$ there exists $m_0$ so that for  $m<m_0$,  $N$ large enough
\begin{equation}\label{cont}
\mathbb E_{\bf a_k}[|\eta_m(z)|]\le \varepsilon
\end{equation}
By \eqref{297125} and since the sign of the imaginary part of $\lambda_N(z)$ is $\op{sgn}_z$, one has $|\eta_m(z)|\leq 4 |\Im z|^{-1}$.  More precisely, for any $K>0$,
we find
$$|\eta_m(z)|\le \frac{4}{\Im z}(mK|\Im z|^{-1}+\one_{\sum_j a_k(j)^2 \ge K})\,.$$
By \eqref{tyu}, we deduce that for any  $\eps>0$ we can choose $m$ small enough so that 
for $N$ large enough
\begin{equation}\label{gamma}\mathbb E_{\bf a_k}[|\eta_m(z)|]\le C \eps\,.\end{equation}
In the following we shall therefore neglect the term $\eta_m(z)$.
By \eqref{297122} and \eqre{297124} we have the following bound: for any $t \neq 0$ and any $z \in \mbb C\setminus \mbb R$ such that $|\Im z|\geq \delta >0$, 
 	\beqy
\label{eq:Domin}	 \Big|\sum_{j=1}^{N-1}\ba_k(j)^2(G_k(z)^2)_{jj} e^{-it \sum \ba_k(j)^2G_k(z)_{jj} }\Big| & \le &  \frac{1}{\delta} \sup_{x\ge 0} x e^{-tx} = \frac {e^{-1}} { t \delta }
	\eeqy
 so that for $M$ large enough 
 and $|\Im z |\geq \delta >0$,
 	\beq
		\E_{\ba_k} \tilde f_{k,m} & = & -i \int_{\op{sgn}_z m}^{\op{sgn}_z M }  \E_{\ba_k} \big[  \partial_z \lambda_N(z) e^{it\lambda_N(z)} \big]\ud t + \varepsilon_{m,M,z,N}
	\eeq
where $\varepsilon_{m,M,z,N}$ is arbitrary small as $M$ is large, uniformly on $N$ and on the randomness. Moreover, by \eqref{eq:Domin} one has 
	\beqy \label{Eq:16052111}
		\E_{\ba_k} \big[ \partial_z \lambda_N(z)  e^{it \lambda_N(z)} \big] =  \frac 1 {it} \partial_z  \big( \E_{\ba_k} \big [  e^{it\lambda_N(z)} \big] \big)
	\eeqy
 Recall that  $\phi_N(\lambda) = \E \big[ \exp( -i \lambda |a_{11}|^2) \big]$, so we have
	\be\la{2071217h44}
		 \E_{\ba_k} \big [  e^{it\lambda_N(z)} \big] =  \E_{\ba_k} \Big [  e^{itz-it \sum_j\ba_k(j)^2G_k(z)_{jj})} \Big]  = e^{itz} \prod_{j=1}^{N-1}\phi_N(tG_k(z)_{jj}).
	\ee
	
	Remind that in \eqref{eq:CVversRho} we have shown in the proof of the last item of Corollary \ref{corconv1}  
	\beqy
		   e^{itz} \prod_{j=1}^{N-1}\phi_N(tG_k(z)_{jj})  \limN e^{itz + \rho_z(t)}\quad a.s.
	\eeqy
	As in the proof of Corollary \ref{corconv1}, since the left hand side is analytic and uniformly bounded, we deduce by Montel's theorem that its convergence entails the convergence of its derivatives.
	We then get by \eqref{2071217h44},  for all $t,z$ so that $t/\Im z>0$, the almost sure convergence
	$$ \partial_z  \bigg( \E_{\ba_k} \Big [  e^{itz-it \sum_j\ba_k(j)^2G_k(z)_{jj})} \Big] \bigg) \limN \partial_z  \Big( e^{itz + \rho_z(t)} \Big).$$
We then obtain by dominated convergence (remember the integrant is uniformly bounded  by \eqref{Eq:16052111} and \eqref{eq:Domin}), 
	\beq
		\E_{\ba_k} \tilde f_{k,m}= 
 -  \int_{\op{sgn}_zm}^{\op{sgn}_zM} \frac {1} t \partial_z e^{  itz+\rho_z(t)}\ud t  + \tilde \varepsilon_{m,M,z,N},
	\eeq 
where $ \tilde \varepsilon_{m,M,z,N}$ converges to zero almost surely as $N$ goes to infinity.
\\
\\
\\
By \eqref{Eq:16052111} and estimate \eqref{eq:Domin}, we have the estimate
	$$ \Big | \frac 1 t  \partial_z e^{  itz+\rho_z(t)} \Big| \leq  \Big( 1 + \frac{e^{-1}}{t \Im z}\Big) e^{- t \Im z }$$
so we can let $M$ going to infinity to obtain the almost sure convergence
	\beq
	\E_{\ba_k} \tilde f_{m,k} \limN 
 -  \int_{\op{sgn}_z m}^{\op{sgn}_z \infty} \frac {1} t \partial_z e^{  itz+\rho_z(t)}\ud t.
 	\eeq
In the L\'evy case, one has $\rho_z(t) = t^{\frac \alpha  2}\rho_z(1)$, and in the exploding moments case, $\big| \frac 1 t \partial _z e^{itz+\rho_z(t)}\big| \leq 1 + \frac{C_2}{\Im z^2}$, so that the integral converges at zero and we obtain
	$$\E_{\ba_k} f_k \limN -  \int_0^{\op{sgn}_z \infty} \frac {1} t \partial_z e^{  itz+\rho_z(t)}\ud t\,.$$
\end{proof}

Of course, this convergence is uniform in $k$ since the law of $G_k(z)$ does not depend on $k$ and an analogous formula is true for  $\E_{\ba_k} f'_k$, replacing $z$ by $z'$.

\subsection{Computation    of $\E_{\ba_k} (f_k\ti f_k'')$}

\begin{lem}\label{Lem:CVdePsi} Almost surely, we have the convergence
\be
	\E_{\ba_k} f_k f_k'' \limN \Psi^u(z,z'):=
 \int_{0}^{\op{sgn}_z \infty}\int_{0}^{\op{sgn}_{z'}\infty}  \frac 1{tt'} \partial^2_{z,z'} e^{itz + it'z' + \rho_{z,z'}^{u}(t,t')} \trm dt \trm dt',
 	\ee
where $ \rho^u_{z,z'}$ is defined in Corollary \ref{corconv1}.\end{lem}

\begin{proof}We shall  start again by using Formula \eqre{88121} for $f_k$, and its analogue for 
	$$f_k''=\f{1+\sum_j\ba_k(j)^2(G'_k(z')^2)_{jj}}{z-\sum_j\ba_k(j)^2G'_k(z')_{jj} }:= \frac{ \partial_z \lambda'_N(z')}{\lambda'_N(z')},$$
 where $G_k'(z')$ is defined as $G_k(z)$, replacing $z$ by $z'$ and the matrix $A$ by the matrix  $A'$ defined by  \eqre{88123}, which gives
	\beqy\la{8.8.12.2}
		f''_k &= & -i \int_0^{\op{sgn}_{z'}\infty}  \partial_z \lambda'_N(z')e^{it' \lambda'_N(z')}\ud t'\\
		& =&    -i \int_{\op{sgn}_z'  m}^{\op{sgn}_z'\infty}\partial_z \lambda'_N(z') e^{it'\lambda'_N(z')} \ud t'   -i \int_{0}^{\op{sgn}_zm} \partial_z \lambda'_N(z') e^{it'\lambda'_N(z')} \ud t' \nonumber 
	\eeqy
	
	The upper bound \eqref{eq:Domin} allows us to bound the first term uniformly by  $\log m^{-1}$ 
	and to truncate the integrals for $\op{sgn}_z t , \op{sgn}_{z'} t' \leq M$. Therefore, 
	up to a small error $\varepsilon$ uniform for $M$ large, $m$ small and provided $|\Im z|, |\Im z'|\geq \delta>0$, we have  
	 	\beqy 
	\mathbb E_{\ba_k}\tilde f_k\ti \tilde f_k'' & = & -\int_{\op{sgn}_z  m}^{\op{sgn}_z M}\int_{\op{sgn}_z'  m}^{\op{sgn}_{z'}M} \mathbb E_{\ba_k}
	\partial_z \lambda_N(z)  \partial_z \lambda'_N(z')  e^{it   \lambda_N(z) + it'  \lambda'_N(z') }\ud t\ud t'  + \varepsilon
	\la{511122}
	 \eeqy
As in the previous case, the upper bound \eqref{eq:Domin} allows us  to write
	\beqy
		\E_{\ba_k} \bigg[ \partial_z \lambda_N(z)  \partial_z \lambda'_N(z')  e^{it   \lambda_N(z) + it'  \lambda'_N(z') } \bigg]  =  - \frac 1{tt'} \partial^2_{z,z'} \bigg(  \E_{\ba_k}\Big[ e^{it   \lambda_N(z) + it'  \lambda'_N(z') } \Big]\bigg)
	\eeqy
Remember that $\phi_N(\lambda) = \E \big[ \exp( -i \lambda |a_{11}|^2) \big]$, so we have
  	\beqy
		\E_{\ba_k} \Big[ e^{it   \lambda_N(z) + it'  \lambda'_N(z') } \Big] = e^{itz+it'z'} \prod_{j=1}^{N-1} \phi_N\big( tG_k(z)_{jj} + t' G'_k(z')_{jj} \big).
	\eeqy
By assumption \eqref{2071216h33} and Corollary  \ref{corconv1},
we have the following almost sure convergence  as $N$ goes to infinity and $\frac kN$ goes to $u$ in $(0,1)$
		$$e^{itz+it'z'} \prod_{j=1}^{N-1} \phi_N\big( tG_k(z)_{jj} + t' G'_k(z')_{jj} \big) \limN e^{itz + it'z' + \rho_{z,z'}^{u}(t,t')}.$$
Almost surely, for any $t,t'$, the map $(z,z') \mapsto e^{itz+it'z'} \prod_{j=1}^{N-1} \phi_N\big( tG_k(z)_{jj} + t' G'_k(z')_{jj} \big) $ is analytic on $\mbb C^{\op{sgn} t} \times \mbb C^{\op{sgn} t'}$ and bounded by one. Hence, with the same arguments as in the previous section, we get almost surely and for any $t,t'$ so that $t/\Im z $ and $t'/\Im z'$ are positive, the uniform convergence for the second derivatives on  compact subsets. The truncations of the integrals can be suppressed as in the previous section, we obtain the almost sure convergence
	\begin{equation*}\E_{\ba_k} f_k\ti f_k''  \limN  \int_{0}^{\op{sgn}_z \infty}\int_{0}^{\op{sgn}_{z'}\infty}  \frac 1{tt'} \partial^2_{z,z'} e^{itz + it'z' + \rho_{z,z'}^{u}(t,t')} \trm dt \trm dt'.
\end{equation*}
\end{proof}

Hence we have proved the convergences \eqre{1481214h45} and \eqre{1481214h45bis}. This completes the proof of Theorem \re{TCLHTRM}.

\section{Proof of Corollary \ref{genf}}\label{part:Corgenf}
Note first that by linearity of the map $f\mapsto Z_N(f)$, to prove the convergence in law of the process, it suffices to prove the convergence in law of $Z_N(f)$ for any fixed $f$. Then,  by Lemma \re{lem151113} and the concentration inequality \eqre{1511131} of the appendix, it suffices to prove the result for $f\in \Cc_c^2(\R)$. For $f\in \Cc_c^2(\R)$, we  use \cite[Eq. (5.5.11)]{alice-greg-ofer} to see that for any probability measure $\mu$
$$\int f(x)\ud \mu(x) =\Re\left(\int \ud x\int \ud y \bar\partial \Psi_f(x,y) \int\frac{1}{t-x-iy} \ud \mu(t) \right)$$
Applying this to the empirical measure of eigenvalues and its expectation, we deduce that
$$Z_N(f)=\Re\left(\int \ud x\int \ud y \bar\partial \Psi_f(x,y) Z_N(x+iy)\right)\,.$$
Note here that $\bar\partial \Psi_f(x,y)$ is supported in a compact set $[-c_0,c_0]\times [0,c_0']$  and is bounded by $c|y|$ for a finite constant $c$.
Hence, the integral is well converging.  We next show that we can approximate it by
$$Z_N^\Delta(f)=\Re\left(\int \ud x \int \ud y \bar\partial \Psi_f(x,y) Z_N(\Delta(x)+i\Delta(y))\right)\,,$$
where $\Delta(x)=\lceil2^k x\rceil2^{-k}$. The random variable $Z_N^\Delta(f)$ is only a finite sum of $Z_N(x+iy)$ and therefore it converges 
in law by Theorem \ref{TCLHTRM},  towards $$Z^\Delta(f)=\Re\left(\int \ud x \int \ud y \bar\partial \Psi_f(x,y) Z(\Delta(x)+i\Delta(y))\right)\,,$$
We finally show the convergence in probability of $Z_N^\Delta(f)-Z_N(f)$ and $Z^\Delta(f) - Z(f)$ to zero by bounding their $L^1$ norms 
$$\mathbb E[|Z_N^\Delta(f)-Z_N(f)|]\le \int_{-c_0}^{c_0}\ud x\int_0^{c_0'} \ud y c|y| \mathbb E[ |Z_N(\Delta(x)+i\Delta(y))-Z_N(x+iy)|]\,.$$
But, by Lemma \ref{le:concspec}, there exists a finite constant $C$ such that for any $\varepsilon>0$
\begin{eqnarray*}
\mathbb E[ |Z_N(\Delta(x)+i\Delta(y))-Z_N(x+iy)|]&\le& C\lf\| \frac{1}{.-x-iy}-\frac{1}{.-\Delta(x)-i\Delta(y)}\ri\|_\textsc{TV}\\
&\le& C\frac{1}{y}\left(\frac{|x-\Delta(x)|+|y-\Delta(y)|}{y}\right)^\varepsilon \,.\end{eqnarray*}
Taking $\varepsilon\in (0,1)$ we deduce that there exists a finite constant $C_\varepsilon$ such that
$$\mathbb E[|Z_N^\Delta(f)-Z_N(f)|]\le C_\varepsilon 2^{-\varepsilon k}\,.$$
Of course the same estimate holds for $\mathbb E[|Z^\Delta(f)-Z(f)|]$.
This implies the desired convergence in probability of $|Z_N^\Delta(f)-Z_N(f)|$ and $|Z^\Delta(f)-Z(f)|$ to zero, and therefore the desired convergence in law of $Z_N(f)$. This proves the corollary. 

\section{Fixed point characterizations}\la{sec:fixedpoint}

In this section, we provide characterizations of the functions $\rho_z$ and $\rho_{z,z'}^u$ involved in the covariance of the limiting process of Theorem \re{TCLHTRM} as fixed points of certain functions.  In fact, we also give an independent 
proof of the existence of such limits, and of Corollary \ref{corconv1}.

\subsection{Fixed point characterization of $\rho_z( \cdot)$: proof of Theorem \ref{Th:TixPtEqRhoz} }
We now prove the fixed point equation for the non random function involved in the Lemma \ref{Lem:CVdeL}, given for $z\in \C^+$ and $\lam>0$ by, 
		$$\rho_z(\lam)= \Nlim \rho_z^N(\lam) = \Nlim \ff{N}\sum_{j=1}^{N}\Phi(\lam G_k(z)_{jj}),$$
where we have proved that this convergence holds almost surely in Corollary \ref{corconv1}. Note however that under the assumptions of Theorem \ref{Th:TixPtEqRhoz}, the arguments below provide another proof of this convergence where we do not have to assume \eqref{CondMom}.
 
 We denote in short $A$ for $A_k$, $G$ for $G_k$ and $\ba$ for $\ba_k$ in the following, and we do not detail  the steps of the proof, which are very similar to those in  \cite{BAGheavytails} and Corollary \ref{corconv1}, but outline them.
 Since we have already seen that $\rho$ is analytic in Corollary \ref{corconv1} we need only to prove the fixed point equation. 
Let $A_1$ be the $N-2\ti N-2$ principal submatrix of $A$ obtained by removing the first row and the first column of $A$, and let $G_1(z):=(z-A_1)^{-1}$. Let $\ba_1$ be the first column of $A$ where the first entry has been removed.
Using first the concentration lemma \re{le:concres}, then exchangeability of the $G(z)_{jj}$'s, then Schur complement formula (see \cite[Lem. 2.4.6]{alice-greg-ofer}),
 and then the fact that we can get rid of the off diagonal terms
by the same argument as in the proof of Lemma \ref{20812} since $\Phi$ is continuous on $\ovl{\C^-}$ (namely via Lemma \re{vanishing_nondiagterms_10713}), we have for all $z\in\bC^+$
  \begin{eqnarray*}
  \mathbb E[\rho^N_z(\lambda)]&=&\E[\Phi(\lam G(z)_{11})]=\mathbb E\big[\Phi(\frac{\lambda}{z-a_{11}-{\bf a}_1^* G_1(z) {\bf a}_1})\big]\\
  &=& \mathbb E\Big[\Phi\Big(\frac{\lambda}{z-
  			 \sum_{j=1}^{N-2}   
			 \ba_{1}(j)^2 G_1(z)_{jj}}\Big)\Big]+o(1).\eeq
  Then we use Hypothesis made at  \eqref{hypcalcconv} to  get   for $\Im z>0$
  	 \beq 
	 	\E[\rho^N_z(\lambda)] 
		&=&\int_0^\infty g(y) \mathbb E[ e^{i\frac{y}{ \lambda} (z-\sum_{j} \ba_{1}(j)^2 G_1(z)_{jj})}
]\ud y+o(1)\\
		&=& \int_0^\infty g(y)e^{i\frac{y}{ \lambda} z } \mathbb E\Big[\prod_{j=1}^{N-1} \phi_N( \frac{y}{ \lambda} G_1(z)_{jj}) \Big] \ud y +o(1).
	\eeq
Using the definition of $\Phi$ and the fact that we assumed that it is bounded on every compact subset (since $\rho^N$ has non positive real part  we can cut the integral to keep $y$ bounded up to a small error, as in the previous sections), we have 
	 \beq
 		 \E[\rho^N_z(\lambda)] 
	 	&=& \int_0^\infty g(y)e^{i\frac{y}{ \lambda} z }  \E\Big[ e^{\frac{1}{N}\sum_j \Phi( \frac{y}{ \lambda} G_1(z)_{jj})}\Big] \ud y+o(1)\\
		&=&\lam\int_0^\infty g(\lam y) e^{i {y}  z }  \E \Big[ e^{\rho_z^{N-1}(y)} \Big]\ud y+o(1).
	\end{eqnarray*}
	Now, notice that by Lemmas \ref{le:prox_mat_ssmat_AD} and \ref{le:concspec},  $\rho_z^{N-1}$ can be replaced by $\mathbb E[\rho_z^{N}]$ in the above formula.
	Moreover, as $\Phi$ is uniformly continuous, so is $\mathbb E[\rho_z^{N}]$, so that we can take limit points  (or we can use
Corollary \ref{corconv1} up to assuming \eqref{CondMom}) and check that they satisfy \eqre{charrhoAD}.

   Let us now prove  that there is a unique  solution to this equation which is analytic in $\mathbb C^+$ and with non positive real part. Suppose that there are two such
    solutions $\rho_z(\lam), \tilde{\rho}_z(\lam)$.   For $z$ fixed, let us define $\Delta(\lam):=|\rho_z(\lam)-\tilde{\rho}_z(\lam)|$. Then for all $\lam$, we have, by the hypothesis made on $g$ in  \eqre{ConditionOng},  
  \beq
		\Del(\lam)& \le & \lam\int_0^\infty|g(\lam y)| e^{-y\Im z }\Del(y)\ud y\\
		& \le & K\bigg( \lam^{\ga+1}\int_0^{\infty} y^\ga e^{-y\Im z }\Del(y)\ud y +  \lam^{\ka+1}\int_{0}^\infty y^{\ka} e^{-y\Im z }\Del(y)\ud y \bigg)\,.
	\eeq
 
 It follows that $\ds I_1:=\int_0^\infty \lam^\ga e^{-\lam\Im z }\Del(\lam)\ud\lam$ and $\ds I_2:=\int_0^\infty \lam^{\ka} e^{-\lam\Im z }\Del(\lam)\ud\lam$ satisfy
\begin{eqnarray*}
 I_1 &\le&  K\bigg( I_1 \int_{0}^\infty \lam^{2\ga+1}e^{-\lam \Im z}\ud \lam + I_2 \int_{0}^\infty \lam^{\ga+ \ka+1}e^{-\lam \Im z}\ud \lam  \bigg),\\
  I_2 & \le &  K\bigg( I_1 \int_{0}^\infty \lam^{\ga+ \ka+1}e^{-\lam \Im z}\ud \lam + I_2 \int_{0}^\infty \lam^{2\ka+1}e^{-\lam \Im z}\ud \lam  \bigg).\end{eqnarray*}
   For $\Im z$ large enough, 
   the integrals above are strictly less that $\frac 1 {2K})$,
    so that
   $I_1=I_2=0$.
It follows that for any fixed $\lam$, $\rho_z(\lam)$ and $\tilde{\rho}_z(\lam)$ are analytic functions of $z$ which coincide for $\Im z$ large enough, hence they are equal.

\subsection{Fixed point characterization of $\rho_{z,z'}(\cdot, \cdot)$: proof of Theorem \ref{Th:TixPtEqRhozz'}}

We now  find a  fixed point system of  equations for the non random function of Corollary \ref{corconv1}. For 
  $\lambda /\Im z + \lam'/\Im z' \ge 0$,  we set 
		\begin{eqnarray}
		\rho^{N,k,1}_{z,z'}(\lam,\lam') &=&  \ff{k-1}\sum_{j=1}^{k-1}\Phi(\lam G_k(z)_{jj} + \lam' G'_k(z')_{jj})\label{rho1}\\
		\rho^{N,k,2}_{z,z'}(\lam,\lam') &=&  \ff{N-k-1}\sum_{j=k}^{N-1}\Phi(\lam G_k(z)_{jj} + \lam' G'_k(z')_{jj}),\label{rho2}
\end{eqnarray}
where we recall that $G_k$ and $G'_k$ are as in \eqref{eq:RhodanIntro2}. To simplify the notations below, as in the previous section, we denote $(G,G')$ instead of $(G_k,G'_k)$, even though their distribution depends on $k$.

  In the sequel we fix, as in Section \re{sec:Step3Levy},  a number  $u\in (0,1)$ and will give limits in the regime where $N\to \infty$, $k\to\infty$ and $k/N\lto u$. We shall then prove  that, under the hypotheses of Theorem \ref{Th:TixPtEqRhozz'} that we assume throughout this section, 
   $(\rho^{N,k,1}_{z,z'} ,\rho^{N,k,2}_{z,z'})$ converges almost surely and that its limit satisfies a fixed point system of equations which has a unique analytic solution with non positive real part. The convergence could be shown with minor modifications of Lemma \ref{Lem:CVparlesTraffics} under assumption \eqref{CondMom}, but we do not need this since we work with stronger assumptions.
  Using the concentration lemma \re{le:concres} (note that $\Phi$ is not Lipschitz but can be approximated by Lipschitz functions uniformly on compacts), it is sufficient to prove the fixed point equation for the expectation of these parameters. Moreover, by exchangeability of the $k$ first entries and $N-k$ last entries
	\begin{eqnarray*}
		\bE\big[ \rho^{N,k,1}_{z,z'} (t,s)\big] & = &   \E\big[\Phi(tG(z)_{11}+s G'(z')_{11})\big] +o(1), \\
		 \bE\big[ \rho^{N,k,2}_{z,z'} (t,s)\big] &=&\E \big[ \Phi(tG(z)_{NN}+s G'(z')_{NN})\big] + o(1).
	\end{eqnarray*}

These functions are analytic in $\Lambda=\{z,z': t/\Im z+s/\Im z'>0\}$ and uniformly bounded continuous in $(t,s)$ (by uniform continuity of $\Phi$ and boundedness of $G(z)$),  and hence tight by Arzela-Ascoli on compacts of $\Lambda$.  We let $\rho^{u,i}_{z,z'},i=1,2$ be a limit point. Since
$\bE\big[ \rho^{N,k,i}_{z,z'} (t,s)\big] $ is uniformly bounded on compacts of $\Lambda$, the limit points $\rho^{u,i},i=1,2$ are analytic by Montel's theorem.
 We assume  for simplicity that both $z$ and $z'$ have positive imaginary parts (in the general case, one only has to replace $\int_{s=0}^{+\infty}\int_{t=0}^{+\infty}$ below by $\int_{s=0}^{\op{sgn}(\Im z)\infty}\int_{t=0}^{\op{sgn}(\Im z')\infty}$). Under Hypothesis made at \eqref{phias}, we can write by Schur complement formula and getting rid of the off diagonal terms (note that all integrals are finite as they contain exponentially decreasing terms)
\begin{eqnarray*}
\E \big[\rho_{z,z'}^{N,k,1}(t,s) \big]&=&  \int_0^\infty\int_0^\infty e^{i\frac vt z+i \frac{v'}s z'} \bE[ e^{-i \frac v t \sum_{\ell} \ba_{1}(\ell)^2 G_1(z)_{\ell\ell} -i \frac{v'}s \sum_{\ell}  \ba_{1}'(\ell)^2 G_{1}'(z)_{\ell\ell} }]\ud\tau(v,v')\\
 &&+\int e^{ i \frac vt}  \bE[ e^{-i \frac v t \sum_{\ell} \ba_{1}(\ell)^2 G_{1}(z)_{\ell\ell} }] \ud\mu(v)+\int e^{ i\frac vt }  \bE[ e^{-i \frac v s \sum_\ell \ba_{1}(\ell)^2 G_k(z')_{\ell\ell} }] \ud\mu(v)+o(1)\\
 &=& \int_0^\infty\int_0^\infty e^{i \frac v tz+i \frac {v'}s z'}  e^{u\E\rho_{z,z'}^{N,k,1}(\frac{v}{t},\frac{v'}{s})+(1-u)\rho_z(\frac{v}{t})+(1-u)\rho_{z'}(\frac{v'}{s})
 }\ud\tau(v,v')\\
&&+\int_0^\infty e^{ i \frac vt}  e^{\rho_z(\frac{v}{t})} \ud\mu(v) +\int_0^\infty e^{ i \frac vs}  e^{\rho_{z'}(\frac{v}{s})} \ud\mu(v)+o(1)\eeq 
and 
\beq
\E \big[\rho_{z,z'}^{N,k,2}(t,s) \big]&=& 
 \int_0^\infty\int_0^\infty e^{i\frac vt z+i \frac{v'}s z'} \bE[ e^{-i \frac v t \sum_{\ell} \ba_{N}(\ell)^2 G_N(z)_{\ell\ell} -i \frac{v'}s \sum_{\ell}  \ba_{N}'(\ell)^2 G_{N}'(z)_{\ell\ell} }]\ud\tau(v,v')\\
 &&+ \int e^{ i \frac vt }  \bE[ e^{-i \frac v t \sum_{\ell} \ba_{N}(\ell)^2 G_{N}(z)_{\ell\ell} }] \ud\mu(v)+\int e^{ i \frac vt }  \bE[ e^{-i \frac v s \sum_\ell \ba_{N}(\ell)^2 G_k(z')_{\ell\ell} }] \ud\mu(v)+ o(1)\\
 &=&  \int_0^\infty\int_0^\infty e^{i \frac v tz+i \frac {v'}s z'}  e^{\rho_z(\frac{v}{t})+\rho_{z'}(\frac{v'}{s})}
\ud\tau(v,v')\\
&&+\int_0^\infty e^{ i \frac vt}  e^{\rho_z(\frac{v}{t})} \ud\mu(v) +\int_0^\infty e^{ i \frac vs}  e^{\rho_{z'}(\frac{v}{s})} \ud\mu(v) +o(1)\\
\end{eqnarray*}
where we used that $ \ba_{1}(\ell)=\ba_1'(\ell)$ for all $\ell\le k$ and are independent for $\ell>k$ whereas $\ba_N(\ell)=\ba_N'(\ell)$   are independent for all $\ell$, that $\rho^N_z$ converges towards $\rho_z$, and that 
		\beq
			 \ff{k-2}\sum_{j=1}^{k-2}\Phi(\lam G_1(z)_{jj} + \lam' G'_1(z')_{jj}) \sim \rho_{z,z'}^{N,k,1}\\
		  \ff{N-k-2}\sum_{j=k}^{N-2}\Phi(\lam G_N(z)_{jj} + \lam' G'_N(z')_{jj}) \sim \rho_{z,z'}^{N,k,2},
		 \eeq
by Lemma \ref{le:prox_mat_ssmat_AD} (by continuity of $\Phi$, it can be approximated by Lipschitz functions).
Hence we find that the limit points $\rho_{z,z'}^{u,1},\rho_{z,z'}^{2}$ of $\rho_{z,z'}^{N,k,1},\rho_{z,z'}^{N,k,2}$ satisfy \eqref{1511121h58} and \eqref{1511121h58BD13}. Moreover,
 $$\rho_{z,z'}^{N,k}=u\rho_{z,z'}^{N,k,1}+(1-u)\rho_{z,z'}^{N,k,2}+o(1)$$
gives 
$$\rho^u_{z,z'}=u\rho^{u,1}_{z,z'}+(1-u)\rho^{u,2}_{z,z'}\,.$$

\noindent{\bf Uniqueness under assumption \eqre{phias}:}  Let  $\rho_{z,z'}^{u,1}(t,s)$ and $\tilde\rho_{z,z'}^{u,1}(t,s)$ be solutions of Equation 
  \eqre{1511121h58} with non positive real parts (note here that $\rho_z(\cdot)$ is given and $\mu$ is so that the above integrals are finite; hence the last two terms in both equations play the role of a finite given function). 
  \begin{eqnarray*}
\Delta(t,s)&:= &|\rho_{z,z'}^{u,1}(t,s)-\tilde \rho_{z,z'}^{u,1}(t,s)|
\\
&\le&2 \int_0^\infty\int_0^\infty e^{-\Im z v t^{-1}- \Im z' v's^{-1}} \frac{\ud |\tau|(v,v') }{\ud v \ud v'} \Delta(\frac{v}{t},\frac{v'}{s}) \ud v \ud v'\\
		& \leq & 2ts \int_0^\infty\int_0^\infty e^{-\Im z v- \Im z' v'} 	K \big( (vt)^\ga \one_{vt\in ]0,1]} + (vt)^{\ka} \one_{vt\in ]1,\infty[}\big) \\
			&& \times \big( ({v's})^{\ga} \one_{v's\in ]0,1]} + ({v's})^{\ka} \one_{v's\in ]1,\infty[}\big)	 \Delta(v,v') \ud v \ud v'\\
		& \leq & 2Kts \big( (ts)^{\ga} I_\Delta(\ga,\ga') +  t^\ga s^{\ka} I_\Delta(\ga,\ka) +  (ts)^{\ga} I_\Delta(\ka,\ga) + (ts)^{\ka} I_\Delta(\ka,\ka)\big),
	\eeq
where for $(\alpha,\alpha') \in \{\ga, \ka\}^2 $, we have set
	$$ I_\Delta(\alpha, \alpha') := \int_0^\infty \int_0^\infty e^{-\Im z v- \Im z' v'}  \Delta(v,v') v^\alpha {v'}^{\alpha'} \trm d v \trm d v'.$$
	We put $I(\alpha,\alpha')=I_1(\alpha,\alpha')$ where $1$ denote the constant function equal to one.
We get after integrating both sides
	\beq
		I_\Delta(\alpha, \alpha')&  \leq  & K \big( I(\alpha+ \gamma+1, \alpha' + \gamma'+1) I_\Delta( \ga, \ga) + I(\alpha+ \gamma+1, \alpha' + \ka'+1) I_\Delta ( \ga, \ka) + \\
		& &  I(\alpha+ \ka+1, \alpha' + \gamma'+1) I_\Delta ( \ka, \ga)+ I(\alpha+ \ka+1, \alpha' + \ka+1) I_\Delta( \ka, \ka)\big).
	\eeq
We consider $\Im z, \Im z'$ large enough so that $I(\alpha+ \beta+1, \alpha' + \beta'+1) < \frac 1{4K}$ for any  $(\alpha,\alpha'), (\beta, \beta') \in\{\ga, \ka\}^2$ to conclude that $\Delta(t,s)$ vanishes then and therefore that $\rho^{u,s}_{z,z'}=\tilde\rho^{u,s}_{z,z'}$ 
for $\Im z$ and $\Im z'$ big enough, $s=1$ or $2$. By analyticity, we conclude that the system of equations \eqref{1511121h58} has a unique analytic solution with non positive real part. The functions $\rho^{N,k,1}$ and $\rho^{N,k,2}$ are tight and their limit points are characterized by fixed point equations, so they actually converge.
\\
\\\noindent{\bf Uniqueness under assumption \eqre{uniquenessassumption2}:} 
 The limit points of $(\rho^{N,k,1}_{z,z'}, \rho^{N,k,2}_{z,z'})$ satisfy \eqref{1511121h58} and \eqref{1511121h58BD13} with $\tau$ given by \eqre{eq:TauLevy}. 
 
 \par To simplify the notations we assume hereafter $\Im z,\Im z'$ non negative. Notice first that if $(g_1,g_2)$ is a limit point of $(\rho^{N,k,1}_{z,z'}, \rho^{N,k,2}_{z,z'})$, $g_2$ is given and the 
$g_i$'s are functions from $(\mathbb R^+)^2$ into $\cL_{\alpha/2}:=\{ -re^{i\theta}\ste r\ge 0,  |\theta|\le \alpha/2\}, i=1,2$, which     are homogeneous  of degree $\frac \alpha 2$, i.e. for any $t,s>0$
	\begin{equation}\label{nk}
		g_i(t,s)=(t^2+s^2)^{\frac{\alpha}{4}}g_i\Big(\frac{t}{\sqrt{t^2+s^2}},\frac{s}{\sqrt{s^2+t^2}}\Big).
	\end{equation}
\par We show that for any $\beta$ in $(\frac \alpha 2, 1)$, the system \eqref{1511121h58} has a unique solution on the set of pair of homogeneous maps $(\mathbb R^+)^2\ra \cL_{\alpha/2}$ of degree $\frac \alpha 2$ that satisfy the $\beta$-H\"older properties, i.e. have finite $\| \cdot \|_{\beta}$ norm given by  
	\begin{equation}\label{eq:NormBeta}
		\| g  \|_{\beta} = \max_{ (u,v) \in S^1_+} | g(u,v) | + \max_{ (u,v) \ne (u',v') \in S^1_+} \frac{ | g(u,v)  - g(u',v') | }{ |(u -u')^2+( v-v')^2 | ^ {  \beta/2} },
	\end{equation}
with $S_+^1=\{s,t\ge 0, s^2+t^2=1\}$. 
\par Proving that the limit points $(\rho_{z,z'}^{u,1}, \rho_{z,z'}^{u,2})$ of $(\rho^{N,k,1}_{z,z'}, \rho^{N,k,2}_{z,z'})$ are $\beta$-H\"older maps for a $\beta>\frac \alpha 2$ allows to conclude the proof of Theorem \ref{Th:TixPtEqRhozz'}. This is the content of the following lemma.

\begin{lem}\la{272132}
For any $z,z'\in\bC^+$,  $u\in [0,1]$, $i\in \{1,2\}$ and  $\beta \in (\alpha/2,(3\alpha/4)\wedge 1)$, 
	$$\|\rho^{u,i}_{z,z'}\|_\beta<\infty.$$
\end{lem}

\bpr  First, since $|G(z)_{\ell \ell}|\le \Im z^{-1}$, we have 
\begin{equation}\label{pol}
\max_{i=1,2} \max_{ (s,t) \in S^1_+}  \Big|\mathbb E[\rho^{N,k,i}_{z,z'}(s,t)]\Big|\le \Big(\frac{1}{\Im z}+\frac{1}{\Im z'}\Big)^{\frac{\alpha}{2}}\,.\end{equation}
 We next show that for any  matrix model so that  $\Phi(x)=-\sigma(ix)^{\frac{\alpha}{2}}$, for  any $2\kappa\in (0,\alpha/2)$
\begin{equation}\label{tot}\limsup_{N\ge 1} \mathbb E\Big[ \big(\sum |a_{1i}|^2\big)^{2\kappa}\Big]<\infty\,.\end{equation}
Indeed, we can write
$$\mathbb E\Big[ \big(\sum |a_{1i}|^2\big)^{2\kappa}\Big]=c \mathbb E\Big[ \int_0^\infty \frac{ 1-e^{-y\sum |a_{1i}|^2}}{y^{1+2\kappa}} \ud y\Big]=\int_0^\infty \frac{ 1-\phi_N( -iy)^N}{y^{1+2\kappa}}\ud y$$
where we have used Fubini for non negative functions. But the above integral is well converging at infinity and we know that $\phi_N$ converges uniformly on $[0,M]$ for all $M$ finite;
hence there exists a finite constant $C$ so that  for $N$ large enough 
$$\sup_{y\in [0,M]}  |y|^{-\frac{\alpha}{2}}N|\phi_N(-iy)-1|\le C$$
which yields
$$\mathbb E\Big[ \big(\sum |a_{1i}|^2\big)^{2\kappa}\Big]\le C \Big(1+ \int_0^M \frac{1-e^{C|y|^{\frac{\alpha}{2}}}}{y^{1+2\kappa}} \ud y\Big)<\infty \mbox{ for } 2\kappa<\alpha/2\,.$$
We next show that this estimate implies the $\beta$-H\"older property.
Indeed, for any $\beta\in [\frac{\alpha}{2},1]$,
 there exists a constant $c = c(\alpha, \beta)$ such that for any $x, y$ in $\bC^-$, 
\begin{equation}\label{fondin}
| x^{\frac \alpha 2}  - y^ {\frac \alpha 2}  | \leq c |x - y | ^ { \beta} \left(  | x | \wedge | y |  \right)^ { \frac  \alpha 2 - \beta }.\end{equation}
Applying this with $x=tG(z)_{jj} +sG'(z')_{jj}$ and $y= t'G(z)_{jj}+s' G'(z)_{jj}$  gives 
$$\Big|\mathbb E\big[\rho^{N,k,i}_{z,z'}(t,s)-\rho^{N,k,i}_{z,z'}(t',s')\big] \Big|\le c K_N \Big(\frac{1}{\Im z^2} +\frac{1}{(\Im z')^2}\Big)^{\frac{\beta}{2}} \big(|t-t'|^2+|s-s'|^2\big)^{\beta/2} $$
with
$$K_N:=\mathbb E\Big[  \big(  | tG(z)_{11} +sG'(z')_{11} | \wedge | t'G(z)_{11} +s'G'(z')_{11}|  \big)^ { -\kappa}\Big]$$
where $\kappa=\beta-\frac{\alpha}{2}>0$. 
It is enough to prove that $K_N$ is uniformly bounded as $\rho^{u,i}_{z,z'}$ is a limit point of $\mathbb E[\rho^{N,k,i}_{z,z'}]$. Note that we may assume that $|s-s'|<1/6$ and $|t-t'|\le 1/6$ 
since otherwise the bound is already obtained by \eqref{pol}. But then this implies that either $t,t'\ge 1/4$ or $s,s'\ge 1/4$. Let us assume $t,t'\ge 1/4$.
Then, we have
$$ |t'G(z)_{11}+s'G'(z')_{11} | \wedge
| tG(z)_{11} +sG'(z')_{11} | \ge t'\wedge t \Im G(z)_{11}\ge \frac{1}{4} |\Im G(z)_{11}|\,.$$
Using Schur formula, we find that
\begin{eqnarray*}
\E \big[|\Im G(z)_{11}|^{-\ka} \big] &=& \E\Bigg[\lf(\frac{(\Im(z-\ba_{1}^* G_1(z) \ba_1))^2+ (\Re( z-a_{11}-\ba_{1}^* G_1(z) \ba_1))^2}{\Im(z-\ba_{1}^* G_1(z) \ba_1)}\ri)^{\ka} \Bigg]\\
 &\le& \frac{2^\ka}{(\Im z)^{3\ka}}\E\bigg[ \Big( C+\big(\sum |a_{1i}|^2\big)^2\Big)^\ka\bigg] \end{eqnarray*}
 so that we deduce that for all $\kappa>0$
 $$
 \E\big[|\Im G(z)_{11}|^{-\kappa} \big ]\le \frac{C+\Im z^\kappa}{\Im z^{3\kappa}} \E\bigg[ \Big( \big(\sum | a_{i1}|^2 \big)^2 +1\Big)^\kappa \bigg]\,.$$
 Equation \eqref{tot} completes the proof by taking $2\kappa= 2(\beta-\frac{\alpha}{2})<\frac{\alpha}{2}$. 
\epr

%
%
We now prove the uniqueness of the solution of \eqref{1511121h58} on the set of functions described above. After some  change of variables, the equation is equivalent to the following:
\begin{eqnarray*} \nonumber
g_1(t,s)&=&  C_\al   \int_0^\infty  \int_0^\infty \int_0^\infty  w^{\alpha/2-1} (w')^{\alpha/2-1} v^{-\alpha/2-1} e^{i W.Z }  \\
&&\bigg(e^{iv T.Z+   u g_1(W+vT)+(1-u) \rho_{z,z'}(W+vT)} -e^{ug_1(W)+ (1-u) g_2(W)}\bigg) \ud w \, \ud w' \ud v\\
\end{eqnarray*}
where we have denoted in short $W=(w,w')$, $Z=(z,z')$, $T=(s,t)$, $\rho_{z,z'}(W)=\rho_z(w)+\rho_{z'}(w')$ and $W.Z, T.Z$ stand for the scalar products.

After the change of variables $w=r\cos(\theta),w'=r\sin(\theta)$, $v=rv'$, we can rewrite this system of equations as
$$
\ g_1= F_{z,z'}^{u}(ug_1+(1-u) \rho_{z,z'})
$$
with, if $T^\dagger=(t,s)$ when $T=(s,t)$,
	$$F^{u}_{z,z'}(g)(T^\dagger)=\int_{v,\tta}  \int_{r=0}^\infty r^{\alpha/2-1}e^{i  r e_\theta. Z}
\left( e^{ir v T.Z+r^{\alpha/2} g(e_\theta +v T)}-e^{r^{\alpha/2}g(e_\theta  )}\right) \mathrm dr \, \mathrm d\mu(v,\theta), $$
where we have denoted $e_\theta=(\cos(\theta),\sin(\theta))$ and
$$d\mu(v,\theta)=C_\alpha \one_{\theta\in [0, \frac{\pi}{2}]} d\theta (\sin2\theta)^{\frac{\alpha}{2}-1} \one_{v\in [0,\infty)}
 v^{-\frac{\alpha}{2}-1}$$
for a constant $C_\alpha$. 
\\
\\The desired uniqueness follows from the following lemma.

\begin{lem} Let $\beta$ in $(\frac \alpha 2 , 1)$. For any $M$ and for $\Im z, \Im z'$ large enough, the map $F^{u}_{z,z'}$ is a contraction mapping on the set $\mathcal C_{M,\beta}$ of homogenous maps  $ g: (\mathbb R^+)^2\ra \cL_{\alpha/2}$ of degree $\frac \alpha 2$ with $\beta$ norm bounded by $M$.
\end{lem}
\begin{proof}
To study the Lipschitz property of $F^u_{z,z'}$ as a function of $g$ in $\mathcal C_{M,\beta}$ for the norm $\beta$, we  first set
 \begin{eqnarray*}
 F^u_{z,z'}(g)(T^\dagger,\tilde T^\dagger)&=&F^u_{z,z'}(g)(T^\dagger)-F^u_{z,z'}(g)(\tilde T^\dagger)\\
 &=&\int \ud \mu(\theta,v )
\int_0^\infty \ud r r^{\alpha/2-1}e^{i  r e_\theta. Z}
\left( e^{ir v T.Z+r^{\alpha/2} g(e_\theta +v T)}-e^{ir v \tilde T.Z+r^{\alpha/2} g(e_\theta +v \tilde T)}\right).\end{eqnarray*}
 We next  bound, for given $g_1,g_2$ in $\mathcal C_{M,\beta}$, $T_1$ in $S^1_+$ and with $T_2$ either in $S^1_+$ or $T_2=0$ (which allows to treat in one time two parts of $\|\cdot\|_\bet$)
 $$\Delta_{F^{u}}= |F^{u}_{z,z'}(g_1)(T_1,T_2)-F^{u}_{z,z'}(g_1)(T_1,T_2)|\,.$$

 For that task, we shall use some technical estimates, with a constant $c$ that may change from  line to line. Remind first the two following bounds from \cite[Lem. 5.7]{charles_alice}
valid for  $\gamma > 0$: there exists a constant $c = c ( \alpha, \gamma)  >0$ such that for all $h,k\in \mbb C^+$  and for  all $x_1, x_2, y_1, y_2 \in \cL_{\frac \alpha 2}$, 
\begin{eqnarray}
&&\qquad J_{\gamma,h}(x_1,x_2,y_1,y_2):= \left| \int_0^\infty   r^{\gamma -1}   e^{ir h }  \left(  \left( e^{r^{\frac \alpha 2} x_1} - e^{r^{\frac \alpha 2}  y_1}  \right) -  \left( e^{r^{\frac \alpha 2 } x_2  }- e^{r^{\frac \alpha 2 } y_2 }\right) \right) \ud  r  \right|\label{qwe1} \\
&  \leq& c   \left( | h |^{- \gamma - \frac \alpha 2 }    | x_1 - x_2 - y_1 + y_2| + | h |^{- \gamma - \alpha}  ( |x_1 - x_2 | + |y_1 - y_2 |   ) ( |x_1 - y_1 | +   |x_2 - y_2 |)  \right).\nonumber \end{eqnarray}
Moreover, for all $h , k  \in \mbb C^+$, $ x, y \in \cL_{\frac \alpha 2}$, 
for   $0 < \kappa \leq 1$, we have
\begin{eqnarray}
K_{\gamma,h,k}(x,y)&:= &\left| \int_0^\infty   r^{\gamma -1} \left( e^{ ir h  }  - e^ { irk} \right)  \left( e^{r^{\frac \alpha 2} x} - e^{r^{\frac \alpha 2}  y} \right) \ud r  \right| \nonumber\\
& \leq& c (  | h |\wedge |k| ) ^{- \gamma -  \frac \alpha 2 - \kappa}    | h - k |^{\kappa} | x - y|,\label{qwe2}
\end{eqnarray}

Moreover, notice that for $(s,t)$ in $S^1_+$, one has $\max (s,t)\ge 1/\sqrt{2}$ and $ \max (\cos\theta,\sin\theta)\ge 1/\sqrt{2}$ and thus
 	$$
		\Im \, T.Z\ge \Im z\wedge \Im z'/\sqrt{2}=:\Delta_{z,z'}\qquad \Im (e_\theta.Z)\ge \Delta_{z,z'},
 	$$
so that $ |ie_\theta.Z+ivT.Z|\;\ge\; \Delta_{z,z'}(1+1_{T\in S_+^1}v)$.
At last, with $a_i = e_\theta+vT_i$ for $i=1,2$, straightforward uses of the $\beta$-norm and the inequalities $|a_i| \leq (1+v)$, $\big| |a_1| -|a_2| \big| \leq v |T_1 -T_2|$, $|a_i| \geq \frac 1 {\sqrt 2}(v \vee 1)$ if ${T_i\in S_1^+}$, $\big| \frac {a_1}{|a_1|} - \frac{a_2}{|a_2|}\big| \leq  (v \vee 1)^{-2} |T_1-T_2|v(1+v)$,  and \eqref{fondin} gives the estimates 	
 	\eq
		 \Big|\sum_{i,j=1}^ 2 (-1)^{i+j} g_i(a_j) \Big|  & \le & c \|g_1- g_2\|_\beta  |T_1-T_2|^\beta f_\beta(v),\\
	 	 \Big(\sum_{i=1}^2 |g_1(a_i)-  g_2(a_i)|\Big) \Big(\sum_{i=1}^2 |g_i(a_1)- g_i(a_2)|\Big) &  \le &  cM \|g_2- g_1\|_\beta |T_1-T_2|^\beta   (1+v)^{\frac \alpha 2} f_\beta(v),
  	\qe
where $f_\beta(v) =  v^\beta \Big(  ( 1+v)^{\frac \alpha 2 + \beta } (v \vee 1)^{-2\beta} +1 \Big)$.
 Using this series of estimates, we find that,
 \begin{eqnarray*}\Delta_{F^{u}}&\le&\int   K_{\frac{\alpha}{2}, e_\theta.Z+vT.Z, e_\theta.Z+v \tilde T.Z}( g(e_\theta+v\tilde T), \tilde g(e_\theta+v\tilde T)) \ud \mu(v,\theta) \\
 &&+ \int J_{\frac \alpha 2, e_\theta.Z+vT.Z}(g(e_\theta+vT), \tilde g(e_\theta+vT), g(e_\theta+v\tilde T), \tilde g(e_\theta+v\tilde T)) \ud \mu(v,\theta) 
 \\
 &\le &  c \|g-\tilde g\|_\beta  \Bigg( \Delta_{z,z'}^{-\alpha}  \one_{\tilde T=0} \int   1_{v\ge 1}  \ud \mu(\theta,v) \\
 &&+    |\tilde T-T|^\beta \Delta_{z,z'}^{-\alpha-\beta} (|z|+|z'|)^\beta  \int  (1+ v \one_{|\tilde T|=1})^{-\beta-\alpha}(1+ \one_{ \tilde T=0, v\leq 1})  v^\beta \ud \mu(\theta,v) \\
&&+  |T-\tilde T|^\beta ( M\Delta_{z,z'}^{-3\alpha/2} +  \Delta_{z,z'}^{-\alpha})  \int v^\beta (1+v)^\beta \Big((v \vee 1)^{-2\beta} +1 \Big) \ud \mu(\theta,v) \Bigg).
 \end{eqnarray*}
While using \eqref{qwe2}, we chose $\gamma=\alpha/2$ and $\kappa=\beta$ when $\tilde T\in S_1$ or $\tilde T=0, v\le 1$, $\kappa=0$ when $\tilde T=0$ and  $v\ge 1$. As the integrals are finite we obtain the desired bound for $g,\tilde g\in \mathcal C_{M,\beta}$
$$\|F^u_{z,z'}(g_1)-F^u_{z,z'}(\tilde g_1)\|_\beta \le C(z,z',M) \|g_1-\tilde g_1\|_\beta,$$
with $ C(z,z',M)<1$ if $ \Im z\wedge \Im z'$ is large enough.
\end{proof} Taking two solutions of \eqref{1511121h58} and \eqref{1511121h58BD13} in ${\mathcal C}_{M,\beta}$, we deduce that they are equal when $\Im z\wedge\Im z'$ is large enough, and thus everywhere by analyticity. 

\section{Proofs of Lemmas \re{lemma10713},  \re{lemma10713LGI} and  \re{lemma10713LGI2}}\la{part:Hyp}
\subsection{Proof of Lemma \re{lemma10713}}
 Let us first treat the case of Wigner matrices with exploding moments.
First and second parts of Hypothesis \re{Hyp:Model}, as well as \eqref{tyu}, are satisfied for $c=c_\eps=0$. Let us then define $\nu_N$ to be the law of $a^2=a_{11}^2$ and $m_N$ to be the measure with density $Nx$ with respect to $\nu_N$, so that for any test function $f$, we have $$\int f\ud m_N=\int Nxf(x)\ud\nu_N(x)=N\E[a^2f(a^2)].$$ Then  for each $k\ge 0$, $$\int x^k\ud m_N(x)=N\E[|a_{11}|^{2(k+1)}]\lto C_{k+1}.$$ As there is a unique measure $m$ on $\R_+$ with $(C_{k+1})_{k\ge 0}$ as sequence of moments, this proves that $m_N$ converges weakly to $m$. Then \eqre{2071216h33} is a consequence that for any continuous bounded function $f$ on $\R_+$, $\int f\ud m_N\lto \int f\ud m$ and the convergence is uniform on uniformly Lipschitz sets of functions (apply this with $f(x)=\f{e^{-i\lam x}-1}{x}$).

Let us now treat the case of \Lvy matrices. Set $t_N:=N^\mu\in (0, \ff{2(2-\al)})$ and define $b:=a\one_{|a|\le t_N}$ and $c:=a\one_{|a|> t_N}$. Then \eqre{107130bis} is obvious by Hypothesis \eqre{ABP09exponent} and the fact that $a_N=N^{1/\al}$ up to a slowly varying factor and \eqre{107131bis} follows directly from Lemma 5.8 of \cite{HTBM12} (in fact, this lemma gives us the right upper bound for the second moment of $b$, which of course implies that it is true for its variance). 
Let us now treat the second part of the hypothesis. Let us fix $\eps>0$ and define $b_\eps=a\one_{|a|\le B}$ (for a constant $B$ which will be specified below) and $c_\eps:=a-b_\eps$. 
For $L$ as in \eqre{ABP09exponent}, we have $$\p(c_\eps\ne 0)=\f{L(a_NB)}{(a_NB)^\al}\sim \f{L(a_N)}{a_N^\al B^\al}\sim\ff{NB^\al}.$$Hence \eqre{limceps} is satisfied if $B$ is chosen large enough. For the convergence of the truncated even moments, see \cite[Sect. 1.2.1]{MAL122}. 
Moreover, \eqre{2071216h33} follows from the results of e.g. Section 8.1.3 of \cite{Bingham-Goldie-Teugels}.
At last, \eqref{tyu} is satisfied for \Lvy matrices by e.g. Section 10 of \cite{BAGheavytails}.

\subsection{Proof of Lemma \re{lemma10713LGI}}

In the case of \Lvy matrices, the expression $$-\si(i\lam)^{\al/2}=\int_{y=0}^{+\infty} C_\al y^{\frac{\alpha}{2} -1}e^{i\f{y}{\lam}}\ud y \qquad\trm{($\lam\in \C^-$)}$$   relies 
 an application of residues formula which gives, for $z\in \C^+$ and $\al>0$,   \be\la{68122}\Ga(\al/2) = -i\int_{t=0}^{+\infty}(-izt)^{\f{\al}{2}-1}e^{itz}z\ud t.\ee

 In the case of Wigner matrices with exploding moment, one first needs to use the following formula, for $\xi\in \C$ with positive real part:\be\la{integraleJ1}1-e^{-\xi}=\int_{0}^{+\infty} \f{J_1(2\sqrt{t})}{\sqrt{t}}  e^{-t/\xi}\ud t ,\ee which is proved in the following way (using \eqre{68122}):$$1-e^{-\xi}=\sum_{p\ge 0}\f{(-1)^{p}}{p!(p+1)!}p!\xi^{p+1}=\sum_{p\ge 0}\f{(-1)^{p}}{p!(p+1)!} \int_{0}^{+\infty} t^p e^{-t/\xi}\ud t=\int_{0}^{+\infty} \f{J_1(2\sqrt{t})}{\sqrt{t}}  e^{-t/\xi}\ud t.$$
 It follows that for $m_N$ the measure introduced in the proof of Lemma \re{lemma10713} above, we have 
 $$N(\phi_N(\lam)-1)=N(\E e^{-i\lam a^2}-1)=-N\E\int_0^{+\infty}\f{J_1(2\sqrt{t})}{\sqrt{t}}e^{-\f{t}{i\lam a^2}}\ud t=\int_0^{+\infty}g_N(y)e^{i\f{y}{\lam}}\ud y$$with \be\la{117131}g_N(y):=-N\f{\E[|a|J_1(2\sqrt{y}|a|)]}{\sqrt{y}}=-N\E[a^2\f{J_1(2\sqrt{ya^2})}{\sqrt{ya^2}}]=\int f_y(x)\ud m_N(x)\ee for $f_y(x):=-\f{J_1(2\sqrt{xy})}{\sqrt{xy}}$. As $m_N$ converges weakly to $m$ and $f_y$ is continuous and bounded, we have $$\ds g_N(y)\lto -\int\f{J_1(2\sqrt{xy})}{\sqrt{xy}}\ud m(x)\,.$$ 
 
 \subsection{Proof of Lemma \re{lemma10713LGI2}}The case of \Lvy matrices is obvious.
 To treat the case of Wigner matrices with exploding moment, first note that by \eqre{integraleJ1}, writing $$e^{-\xi-\xi'}-1=(e^{-\xi}-1)(e^{-\xi'}-1)+(e^{-\xi}-1)+(e^{-\xi'}-1),$$ we have, for any $\xi,\xi'\in \C$ with positive real parts,  \beqy\la{integraleJ12}e^{-\xi-\xi'}-1&=&\iint_{(\R_+)^2}  \f{J_1(2\sqrt{t})J_1(2\sqrt{t'})}{\sqrt{tt'}}  e^{-\f{t}{\xi}-\f{t'}{\xi'}}\ud t\ud t' \\ \nonumber &&\qquad\qquad -\int_{0}^{+\infty} \f{J_1(2\sqrt{t})}{\sqrt{t}}  e^{-t/\xi}\ud t-\int_{0}^{+\infty} \f{J_1(2\sqrt{t'})}{\sqrt{t'}}  e^{-t'/\xi'}\ud t'\eeqy
 As a consequence, for $\lam,\mu\in \C^-$, 
 \beq N(\phi_N(\lam+\mu)-1)&=&N\E[e^{-i\lam a^2-i\mu a^2}-1]\\
  &=&\iint_{(\R_+)^2}  g_N(u,u') e^{i\f{u}{\lam}+i\f{u'}{\mu}}\ud u\ud u' \\ && + \int_{0}^{+\infty} g_N(u)  e^{iu/\lam}\ud u+ \int_{0}^{+\infty}g_N(u')  e^{iu'/\mu}\ud u'
 \eeq
 with $g_N(u)$ defined by \eqre{117131} and $$g_N(u,u'):=\int\f{J_1(2\sqrt{ux})J_1(2\sqrt{u'x})}{\sqrt{uu'}}\ud m_N(x).$$ Then one concludes as for the proof of Lemma \re{lemma10713LGI}.

\section{Appendix}
\subsection{Concentration of random matrices with independent rows and linear algebra lemmas}\la{sec:concentration}
This section is mostly a reminder of results from \cite{charles_alice} and \cite{BCC}.

The total variation norm of $f:\bR \to\C $ is
\be\la{1611131}
\|f \|_\textsc{TV}:=\sup \sum_{k \in \z} | f(x_{k+1})-f(x_k) |, 
\ee
where the supremum runs over all sequences $(x_k)_{k \in \z}$ such that
$x_{k+1} \geq x_k$ for any $k \in \z$. If $f = \one_{(-\infty,s]}$ for
some real $s$ then $\|f \|_\textsc{TV}=1$, while if $f$ is absolutely continuous  (hence almost everywhere differentiable and equal to the integral of its derivative) with  derivative in
$\mathrm{L}^1(\bR)$, we get
\be\la{1611132}
\|f \|_\textsc{TV}=\int |f'(t)|\,\ud t.
\ee
The next lemma is an easy consequence of Cauchy-Weyl interlacing Theorem. It is an ingredient of the proof of Lemma \ref{le:concspec}. 
\begin{lem}[Interlacing of eigenvalues]\label{le:Cauchy}
Let $A$ be an $N \times N$ hermitian matrix and $B$ a principal minor of $A$. Then for any $f:\bR\to\bC$ such that
  $\| f \|_{TV} \leq 1$ and $\lim_{|x| \to \infty} f(x) = 0$, 
$$
\left|\sum_{i=1}^N f ( \lambda_i ( A) )  - \sum_{i=1}^{N-1} f ( \lambda_i ( B) ) \right|\leq  1.
$$
\end{lem}

\beg{lem}\la{le:prox_mat_ssmat_AD}
 Let $A_1,A_2$ be $N \times N$ random Hermitian matrices and $\tilde A_1,\tilde A_2$ be $n-1 \times n-1$ matrices obtained from $A_1$ and $A_2$ respectively by removing the $\ell$-th row and column, for some $\ell\in \{1,\ldots, N\}$. Let $z,z'\in\C$, $t,t'\in \R$ so that $\Im z t>0$ and $\Im z' t'>0$ and set 
  $G=(z-A_1)^{-1}$, $G'=(z'-A_2)^{-1}$ and $\tilde G=(z-\tilde A_1)^{-1}$, $\tilde G'=(z'-\tilde A_2)^{-1}$. Then, for any function $f$ on $B_{z,z',t,t'}:=\{ g\in \C^-\ste |g|\le C(z,z',t,t')\}$
  with $C(z,z',t,t')=t(\Im z)^{-1}+t'(\Im z')^{-1}$, we have  
 	\be\la{5111216h24}
		\lf|\ff{N}\sum_{k=1}^Nf(t G_{kk} + t' G'_{kk})-\ff{N}\sum_{k=1}^{N-1}f(t \tilde G_{kk} + t' \tilde G'(kk)) \ri|\le \frac{C(z,z',t,t')}{N} \|f\|_{\op{Lip}}+\frac{\|f\|_\infty}{N},
	\ee
	where $\|f\|_{\op{Lip}}:=\sup_{x\ne y}\f{|f(y)-f(x)|}{|y-x|}$ and $\|f\|_\infty:=\sup_x |f(x)|$, both supremums running over the elements of $B_{z,z',t,t'}$.
\en{lem}

\bpr  The proof is similar to  \cite[(91)]{charles_alice}.
We denote by $\bar A_1,\bar A_2$ the $N\ti N$ matrices whose entries are the same as $A_1,A_2$ except for the $\ell$th rows and column which have zero entries.
We denote $\bar G,\bar G'$ the corresponding Stieltjes transform.
Then, $\tilde G,\tilde G'$ equal $\bar G,\bar G'$ except at the $\ell$th column and row (where it is equal to $z^{-1} 1_{i=j=k}$).
Therefore, noting  $\bar M=t\bar G(z)+t'\bar G'(z')$ and $\tilde M=t\tilde G(z)+t' \tilde G'(z)$,
we conclude that
$$\lf|\ff{N}\sum_{k=1}^Nf(\tilde M_{kk})-\ff{N}\sum_{k=1}^{N-1}f(\bar M_{kk}) \ri|\le  \frac{\|f\|_\infty}{N}\,.$$
Moreover, let $M=tG(z)+t'G'(z')$  and note that$A_1-\bar A_1$ and $A_2-\bar A_2$ have rank one so that
 $M-\bar M$ has rank one. On the other hand it is bounded uniformly by 
 $C=C(z,z',t,t')$. Hence, we can write $M-\bar M= c u u^*$ with a unit vector $u$ and $c$ bounded by $C$.
 Therefore, since  $f$ is  Lipschitz, 
 \begin{eqnarray*}
 \lf|\ff{N}\sum_{k=1}^Nf(M_{kk})-\ff{N}\sum_{k=1}^{N-1}f(\bar M_{kk}) \ri|&\le&  \f{\|f\|_{\op{Lip}}}{N}\sum_{k=1}^N |M_{kk}-\bar M_{kk}|\le   \f{\|f\|_{\op{Lip}}}{N}\sum_{k=1}^N 
 C\langle e_k,u\rangle^2\\
 &=&\frac{C \|f\|_{\op{Lip} }}{N}\,.\end{eqnarray*}
 \epr

\begin{lem}[Concentration for spectral measures \cite{BCC2}]\label{le:concspec}Let $A$ be an $N\times N$ random Hermitian matrix. Let us assume that the
  vectors $(A_i)_{1 \leq i \leq N}$, where $A_i := (A_{ij})_{1 \leq j \leq i}
  \in \bC^i$, are independent. Then for any measurable $f:\bR\to\bC$ such that
   $\bE |\int\!f\,d\mu_A |<\infty$, and every
  $t\geq0$,
  \[
  \bP \left( \left | \int\!f\,d\mu_A -\bE\int\!f\,d\mu_A  \right| \geq t \right) %
  \leq  \exp\left({-\frac{Nt^2 }{2\| f \|_\textsc{TV}^2}}\right).
  \]
  As a consequence, \be\la{1511131}\E\lf[\left | \int\!f\,d\mu_A -\bE\int\!f\,d\mu_A  \right|^2\ri]\le 2\f{\| f \|_\textsc{TV}^2}{N}.\ee
\end{lem}

\begin{lem}[Concentration for the diagonal of the resolvent]\label{le:concres}
 a)  Let $A$ be an $N\times N$ random Hermitian matrix and consider its resolvent matrix $G (z) = (A - z)^{-1} $, $z \in \bC_+$. Let us assume that the
  vectors $(A_i)_{1 \leq i \leq N}$, where $A_i := (A_{ij})_{1 \leq j \leq i}
  \in \bC^i$, are independent. Then for any $f:\bC^-\to\bR$ such that
  $\| f \|_{\op{Lip}} \leq 1$, and every
  $t\geq0$,
  \[
  \bP \left( \left|\frac 1 n  \sum_{k=1}^N  f (G (z) _{kk} ) -\bE \frac 1 N  \sum_{k=1}^N  f (G (z) _{kk} )  \right|  \geq t \right) %
  \leq 2 \exp\left({-\frac{N  \Im (z)^{2 } t^2}{8}}\right).
  \]

 b)  Let $A'$ be  an $N\times N$ self-adjoint matrices so that $A_{ij}'=A_{ij}, i\wedge j\le k$ and  $(A_{ij}')_{j\ge k+1,i\ge k+1}$ is independent from $(A_{ij})_{j\ge k+1, i\ge k+1}$  but with the same distribution.  Let $G(z)=(z-A)^{-1}$ and $G'(z)=(z-A')^{-1}$ and set
  for a Lipschitz function $f$ on $\overline{ \mathbb C}^{-}$,
  \begin{eqnarray*}
  \rho^{N,k,1}_{z,z'}({\lambda,\lambda'})[f]:&=&\frac{1}{k}\sum_{\ell=1}^k f(\lambda G(z)_{\ell\ell}+ \lambda'G(z')_{\ell\ell})\\
  \rho^{N,k,2}_{z,z'}(\lambda,\lambda')[f]&=&\frac{1}{N-k}\sum_{\ell=k+1}^N f(\lambda G(z)_{\ell\ell}+ \lambda'G(z')_{\ell\ell})\end{eqnarray*}
  Then, for $\lambda/\Im z\ge 0$, $\lambda'/\Im z'\ge 0$, we have for all $\delta\ge 0$, $s\in\{0,1\}$,
  \be\la{27213}\bP\left(\left|\rho^{N,k,s+1}_{z,z'}({\lambda,\lambda'})[f]-\bE[\rho^{N,k,s+1}_{z,z'}({\lambda,\lambda'})[f]\right|\ge\delta\right)\le 2e^{-\frac{\delta^2 ((k-1)^{1-s} +(N-k-1)^s)}{ 8\|f\|_{\rm Lip}^2 C(\lambda,\lambda',t,t')^2}}
  \ee
  with $$C(\lambda,\lambda',z,z')=\frac{2\lam}{\Im z} +\frac{2t\lam'}{\Im z'} \,.$$\end{lem}

%

 \bpr The first  point is proved  as in \cite[Lemma C.3]{charles_alice}.
 We outline the proof of the second point which is very similar to 
 \cite[Lemma C.3]{charles_alice}. We concentrate on $\rho^{N,k,1}_{\lambda,\lambda'}$,
 the other case being similar. By Azuma-Hoefding inequality, it is sufficient to show that
 $$X_p:=\mathbb E[  \rho^{N,k,1}_{z,z'}(\lambda,\lambda')[f]|{\mathcal F}_p]-\mathbb E[  \rho^{N,k,1}_{z,z'}(\lambda,\lambda')[f]|{\mathcal F}_{p-1}]$$
 is uniformly bounded by $\|f\|_{\rm Lip}C(\lambda,\lambda',t,t') k^{-1}$. Here ${\mathcal F}_p$ is the $\sigma$-algebra generated with respect 
 to the $p$ first column (and row) vectors. Note that 
 $X_p$ can be written as the  conditional expectation of the difference of the parameter $f$ evaluated at two sets $A,A'$ and $\tilde A, \tilde A'$ 
 which differ only at the $p$-th vector column (and row).  Hence, we may follow the proof of  Lemma \ref{le:prox_mat_ssmat_AD}
to conclude that $$|X_p|\le \frac{\|f\|_{\rm Lip} }{k} \sum_{\ell=1}^k |(M-\tilde M)_{\ell\ell}|= \frac{|c|\|f\|_{\rm Lip} }{k} \sum_{\ell=1}^k <u, e_\ell>^2
 \le \frac{|c|\|f\|_{\rm Lip} }{k} \,.$$
 \epr

Let $H=[h_{ij}]$ be an $N\ti N$ Hermitian matrix and  $z\in \C\bck\R$. Define $G:=(z-H)^{-1}$. 
 
\beg{lem}[Difference of traces of a matrix and its major submatrices]\la{lat} Let $H_k$ be the submatrix of $H$ obtained by removing its $k$-th row and its $k$-th column and set $G_k:=(z-H_k)^{-1}$. Let also $\ba_k$ be the $k$-th column of $H$ where the $k$-th entry has been removed. Then \be\la{307121}\Tr(G)-\Tr(G_k)=\f{1+\ba_k^*G_k^{2}\ba_k}{z-h_{kk}-\ba_k^*G_k\ba_k}.\ee Moreover, \be\la{307122}|\Tr(G)-\Tr(G_k)|\le \pi|\Im z|^{-1}.\ee \en{lem}

\bpr For \eqre{307121}, see \cite[Th. A.5]{bai-silver-book}.
For  \eqre{307122}, see Lemma \re{le:Cauchy}.
\epr

\beg{lem}\la{lem267121} With the notation introduced above the previous lemma, for each $1\le j\le N$,  \be\la{297122}\Im z \ti \Im G_{jj}<0,\ee  \be\la{297124}  |\Im z|\ti |(G^2)_{jj}| \le |\Im G_{jj}|\le |\Im z|^{-1} \ee
and for any $\ba=(\ba_1,?\ld, \ba_N)\in \C^N$, 
\be\la{297125}
\lf|\f{1+\sum_j|\ba_j|^2(G^2)_{jj}}{z-\sum_j|\ba_j|^2G_{jj} }\ri|\le 2|\Im z|^{-1}.\ee 
\en{lem}

\bpr Set $z=x+i y$, $x,y\in \R$. Let $\lam_1, \ld,\lam_N$ be the eigenvalues of $H$, associated with the orthonormalized collection of  eigenvectors $\ub_1, \ld, \ub_N$. Let also $\bef_j$ devote the $j$th vector of the canonical basis. Then \eqre{297122} and \eqre{297124} follow directly from the following: \be\la{297121}\quad|(G^2)_{jj}|=| \sum_{k=1}^N\f{|\lan \bef_j, \ub_k\ran|^2}{(z-\lam_k)^2}|\le \sum_{k=1}^N\f{|\lan \bef_j, \ub_k\ran|^2}{(x-\lam_k)^2+y^2} \qquad \Im G_{jj}=-\sum_{k=1}^N\f{|\lan \bef_j, \ub_k\ran|^2y}{(x-\lam_k)^2+y^2}\ee
Let us now prove \eqre{297125}. By \eqre{297122}, we know that $\Im z$ and $-\Im G_{jj}$ have the same sign, so 
  \be\la{297123} \lf|\ff{ z-\sum_j|\ba_j|^2G_{jj} }\ri|\le  \ff{|\Im(z-\sum_j|\ba_j|^2G_{jj})|} \le |\Im z|^{-1}.\ee 
  Hence it remains only to prove \eqref{297125}. This is a direct consequence of   \eqref{297122} and \eqref{297124} which imply the second and last inequality
   $$ \lf|\f{\sum_j|\ba_j|^2(G^2)_{jj}}{z-\sum_j|\ba_j|^2G_{jj}}\ri|\le  \f{\sum_j|\ba_j|^2|(G^2)_{jj}|}{|\Im (z- \sum_j|\ba_j|^2 G_{jj}(z))| }\le  \f{\sum_j|\ba_j|^2|(G^2)_{jj}|}{\sum_j|\ba_j|^2|\Im G_{jj}| }\le \frac{1}{\Im z}\,. $$
 \epr

\subsection{Vanishing of non diagonal terms in certain quadratic sums of random vectors}
Let $\|M\|_{\op{op}}$ denote the operator norm of a complex matrix with respect to the canonical 
Hermitian 
norms.

\beg{lem}\la{vanishing_nondiagterms_10713} For each $N\ge 1$, let  $(a_1,\ld, a_N)$ be
a family of i.i.d. copies of an random variable $a$  \st $a$ can be decomposed into $a=b+c$ with  $b,c$ \st $b$ is centered and  \eqre{107130bis}, \eqre{107131bis} of Hypothesis \re{Hyp:Model} are satisfied. Let also $B_N$ be a non random $N\ti N$ matrix \st    $N^{-1}\Tr(B_NB_N^*)$ is bounded. Then we have the convergence in \pro $$X:=\sum_{i\ne j} a_iB_{ij}a_j\lto 0.$$
\en{lem}

\bpr For each $i$, let $a_i=b_i+c_i$ be the decomposition of $a_i$ corresponding to $a=b+c$. Set $X^b:=\sum_{i\ne j} b_iB_{ij}b_j$ and define the event $E_N:=\{\forall i, c_i=0\}$. Note that when $E_n$ holds, $X=X^b$. But by \eqre{107130bis} and the union bound, $\p(E_N)\lto 1$, so that it suffices to prove that $X^b$ converges in \pro to zero, which follows from the 
fact that   its second moment tends to zero. Indeed, by independence of the $b_i$'s and the fact that they are centered, its second moment is $$\E\sum_{i\ne j} b_i^2(B_{ij}^2+B_{ij}B_{ji})b_j^2\le 2N\var(b)^2\frac{1}{N}\Tr(B_NB_N^*) \,.$$
\epr

\subsection{CLT for martingales}

Let $(\mc{F}_k )_{k\ge 0}$ be a filtration \st $\mc{F}_0 =\{\emptyset,\Omega\}$ and let $(M_k )_{k\ge 0}$ be a square-integrable complex-valued martingale starting at zero with respect to this filtration. For $k\ge 1$, we define the random variables $$Y_k:=M_k-M_{k-1}\qquad v_k:=\E[|Y_k|^2\,|\,\mc{F}_{k-1} ]\qquad  \tau_k:=\E[Y_k^2\,|\,\mc{F}_{k-1} ]$$ and we also define $$v:=\sum_{k\ge 1}v_k\qquad  \tau:=\sum_{k\ge 1} \tau_k\qquad L(\eps):=\sum_{k\ge 1} \E[|Y_k|^2\one_{|Y_k|\ge \eps}].$$

Let now everything depend on a parameter $N$, so that $\mc{F}_k=\mc{F}_k(N), Y_k=Y_k(N),v=v(N),\tau=\tau(N),  L(\eps)=L(\eps, N), \ldots$

Then we have the following theorem. It is proved in the real case at \cite[Th. 35.12]{Billingsley}.  The complex  case can be deduced noticing that for $z\in \C$, $\Re(z)^2, \Im(z)^2$ and $\Re(z)\Im(z)$ are linear combinations of $z^2$, $\ovl{z}^2$, $|z|^2$.
\beg{Th}\la{thconvmart}Suppose that for  some  constants $v\ge  0, \tau\in \C$, we have the convergence in probability  for any $\eps>0$
$$v(N)\Ninf v\qquad \tau(N)\Ninf \tau, \qquad L(\eps, N)\Ninf 0.$$ Then we have the convergence in distribution $$\sum_{k\ge 1} Y_k{(N)}\Ninf Z,$$ where $Z$ is a centered complex Gaussian variable \st $\E(|Z|^2)=v$ and $ \E(Z^2)=\tau$. 
\en{Th}

\subsection{Extension of CLTs for random matrices}The following lemma is borrowed from the paper of Shcherbina and Tirozzi \cite{tirozzi}, except that we do not require, in the hypotheses here, $V$ to be continuous, which is very useful in our case.  
 \beg{lem}\la{lem151113}Let, for each $N$, $(\xi_i^{(N)})_{i=1}^N$ be a collection  of $\R^d$-valued random variables. For each $\vfi:\R^d\to \R$, set $$Z_N(\vfi):=u_N\sum_{i=1}^N (\vfi(\xi_i^{(N)})-\E[\vfi(\xi_i^{(N)})]),$$where  $u_N$ is a sequence of real numbers. We make the following hypotheses : \bgt\ite  For any $\vfi$ in a certain normed subspace $(\Lc, \|\cdot\|)$ of the set of functions $\R^d\to \R$, \be\la{maj_var}\E[Z_N(\vfi)^2]\le \|\vfi\|^2.\ee
 \ite  There is a dense subspace $\Lc_1\subset \Lc$ and a   quadratic form $V : \Lc_1\to \R_+$ \st for any 
 $\vfi\in \Lc_1$, we have the convergence in distribution \be\la{conv322312}Z_N(\vfi)\Ninf \mc{N}(0, V(\vfi)).\ee
 \ent
 Then $V$ is continuous on $\Lc_1$, can be (uniquely) continuously extended to $\Lc$ and \eqre{conv322312} is true for any $\vfi\in \Lc$. 
 \en{lem}

 \bpr This is exactly Proposition 4 of \cite{tirozzi}, except that in \cite{tirozzi}, the hypotheses include the continuity of $V$. Let us prove that 
  the hypotheses made here imply that $V$ is continuous on $\Lc_1$. 
  For any $\vfi\in \Lc_1$, $V(\vfi)$ is the second moment of the limit law of $Z_N(\vfi)$. Hence $$V(\vfi)\le \liminf_{N\to\infty}\E[Z_N(\vfi)^2]\le \|\vfi\|^2.$$ This proves that the quadratic form $V$ is continuous.\epr
 
\subsection{On the Hadamard product of Hermitian matrices}

\begin{propo}\label{Prop:Existence} Let $A_1 \etc A_p$ be $N$ by $N$ Hermitian random matrices whose entries have all their moments. Then, there exists a family of random variables $(a_1 \etc a_p)$ whose joint distribution is given by:
	\begin{equation*}
		\mathbb E[ a_1^{n_1} \dots a_p^{n_p}  ] = \mathbb E\Bigg[  \frac 1 N \mathrm{Tr} [   A_1^{n_1} \circ \dots \circ A_p^{n_p} ] \Bigg], \ \forall n_1 \etc n_p \geq 0,
	\end{equation*}
where $\circ$ denotes the Hadamard (entry-wise) product.
\end{propo}

\begin{proof}

{\it Step 1.} We first assume that the matrices are deterministic and have distinct eigenvalues. By the spectral decomposition, for $j=1\etc p$, we have $A_j = \sum_{i=1}^N \lambda_{j,i} u_{j,i} u_{j,i}^*$ where $\Lambda_j = (\lambda_{j,i})_{i=1\etc N}$ is the family of eigenvalues of $A_j$ arranged in increasing order, and $U_j=(u_{j,i})_{i=1\etc N}$ is the family of associated eigenvectors. For any $n_1 \etc n_p \geq 0$, one has
	\eq 
		\frac 1 N \Tr  [   A_1^{n_1} \circ \dots \circ A_p^{n_p} ] 
			& = & \frac 1 N \sum_{k=1}^N \big( A_1^{n_1} \big)(k,k) \dots  \big( A_p^{n_p} \big)(k,k)\\
			& = & \frac 1 N \sum_{k=1}^N \bigg(  \sum_{i_1=1}^N \lambda_{1,i_1}^{n_1} u_{1,i_1} u_{1,i_1}^* \bigg)(k,k) \dots  \bigg( \sum_{i_p=1}^N \lambda_{p,i_p}^{n_p} u_{p,i_p} u_{p,i_p}^*\bigg)(k,k)\\
			& = &  \frac 1 {N^p} \sum_{i_1 \etc i_p=1}^N   \lambda_{1,i_1}^{n_1}  \dots  \lambda_{p,i_p}^{n_p}
		 	\times \bigg ( N^{p-1}  \sum_{k=1}^N \big| u_{1,i_1}(k) \big|^2  \dots  \big| u_{p,i_p}(k) \big|^2  	\bigg).
	\qe
For $j=1\etc p$, we set $\textrm d \mu_{\Lambda_j} = \frac 1 N \sum_{i=1}^N \delta_{\lambda_{j,i}}$ the empirical eigenvalues distributions of $A_j$. We denote $F_{\Lambda_j}(t) = \mu_{\Lambda_j}\big( (-\infty, t]\big)$ the cumulative function of $\textrm d \mu_{\Lambda_j}$. Since the eigenvalues of the matrices are distinct, one has $F_{\Lambda_j}(\lambda_{j,i}) =\frac i N$ for any $i=1\etc N$, $j=1\etc p$. Hence, we have
	$$	\frac 1 N \Tr  [   A_1^{n_1} \circ \dots \circ A_p^{n_p} ] 
			 =  \int_{\mathbb R^p} \lambda_1^{n_1} \dots \lambda_p^{n_p}  f_N\big((\lambda_j, \Lambda_j, U_j)_{j=1\etc p}\big) \textrm d \mu_{\Lambda_1}(\lambda_1) \dots \textrm d \mu_{\Lambda_p}(\lambda_p)  ,
	$$

\noindent where $ f_N\big((\lambda_j, \Lambda_j, U_j)_{j=1\etc p}\big) = \bigg ( N^{p-1}  \sum_{k=1}^N \prod_{j=1}^p \big| u_{j,(N F_{\Lambda_j}(\lambda_j))}(k) \big|^2  \bigg).$ Hence, a family of random variables $(a_1 \etc a_p)$ as in the proposition exists and its joint distribution has density $ f_N\big(( \cdot, \Lambda_j, U_j)_{j=1\etc p}\big)  $ with respect to $\mu_{\Lambda_1} \otimes  \dots \otimes \mu_{\Lambda_p}$.
\\
\\\noindent {\it Step 2.} We now assume that $(A_1 \etc A_p)$ are random and that their joint distributions have a density with respect to the Lesbegue measure on $\mathcal H_N^p$, where $\mathcal H_N$ is the space of Hermitian matrices of size $N$. In particular, the matrices have almost surely $N$ distinct eigenvalues, see \cite{DEI}. The spectral decompositions of the previous step are measurable (see \cite[Section 5.3]{DEI}) and, with the above notations $(\Lambda_j, U_j)$ for eigenvalues and eigenvectors of $A_j$, we can write the joint distribution of $(A_1 \etc A_p)$ in the form $g_N\big((\Lambda_j, U_j)_{j=1\etc p} \big) \textrm d \mu_{\Delta_N}(\Lambda_1) \dots  \textrm d \mu_{\Delta_N} (\Lambda_p)\textrm d \mu_{\mathcal U_N}(U_1) \dots \textrm d \mu_{\mathcal U_N}(U_p)$. The symbol $\mu_{\Delta_N}$ denotes the Lebesgue measure on $\Delta_N = \{(x_1\etc x_N) | x_1< \dots < x_N\}$ and $\mu_{\mathcal U_N}$ is the Haar measure on the set $\mathcal U_N$ of unitary matrices of size $N$. For any $n_1 \etc n_p \geq 0$, one has
	\eq 
		 \lefteqn{\esp \bigg[ \frac 1 N\Tr  [   A_1^{n_1} \circ \dots \circ A_p^{n_p} ]  \bigg]  }\\
		 & = & \int_{\Delta_N^p \times \mathcal U_N^p} \frac 1 {N^p} \sum_{i_1 \etc i_p=1}^N \lambda_{1,i}^{n_1} \dots   \lambda_{p,i}^{n_p}  f_N\big((\lambda_{j,i}, \Lambda_j, U_j)_{j=1\etc p} \big) \\
		 && \ \ \ \ \times \ g_N\big((\Lambda_j, U_j)_{j=1\etc p}\big)\textrm d \mu_{\Delta_N}(\Lambda_1) \dots  \textrm d \mu_{\Delta_N} (\Lambda_p)\textrm d \mu_{\mathcal U_N}(U_1)  \dots \textrm d \mu_{\mathcal U_N}(U_p)\\
		 & = & \frac 1 {N^p} \sum_{i_1 \etc i_p=1}^N \int_{\Delta_N^p} \lambda_{1,i}^{n_1} \dots \lambda_{p,j}^{n_p} h_N\big((\lambda_{j,i}, \Lambda_j)_{j=1\etc p} \big) \textrm d \mu_{\Delta_N}(\Lambda_1) \dots \textrm d \mu_{\Delta_N} (\Lambda_p),
	\qe

\noindent where $f_N$ is as in the previous step and
	\eq 
		\lefteqn{ h_N\big((\lambda_{j,i}, \Lambda_j)_{j=1\etc p} \big)} \\
		& = & \int_{\mathcal U_N^p}   f_N\big((\lambda_{j,i}, \Lambda_j, U_j)_{j=1\etc p} \big) g_N\big((\Lambda_j, U_j)_{j=1\etc p}\big)   \textrm d \mu_{\mathcal U_N}(U_1) \etc  \textrm d \mu_{\mathcal U_N}(U_p).
	\qe

\noindent For any $i_1 \etc i_p=1\etc N$, we have 
	\eq
		\lefteqn{ \int_{\Delta_N^p} \lambda_{1,i}^{n_1} \dots \lambda_{p,i}^{n_p} h_N\big((\lambda_{j,i}, \Lambda_j)_{j=1\etc p} \big)  \textrm d \mu_{\Delta_N}(\Lambda_1) \etc   \textrm d \mu_{\Delta_N}(\Lambda_p)}\\
		& = & \int_{\mathbb R^p} \lambda_1^{n_1} \dots \lambda_p^{n_p}	 h_N^{(i_1 \etc i_p)}(\lambda_1 \etc \lambda_p) \mathrm d\lambda_1 \etc  \mathrm d\lambda_p,
	\qe

\noindent where $h_N^{(i_1 \etc i_p)}(\lambda_1 \etc \lambda_p)$ is obtained by integrating with respect to the variables $\lambda_{k_1} \etc \lambda_{k_p}$ for $k_1\neq i_1 \etc k_p \neq i_p$. We finally obtain
	\eq
		 \esp \bigg[ \frac 1 N \Tr  [   A_1^{n_1} \circ \dots \circ A_p^{n_p} ]  \bigg]  & = & \int_{\mathbb R^p} \lambda_1^{n_1} \etc \lambda_p^{n_p}  \bar h_N(\lambda_1 \etc \lambda_p) \mathrm d \lambda_1 \etc  \mathrm d \lambda_p,
	\qe

\noindent where $ \bar h_N = \frac 1 {N^p} \sum_{i_1 \etc i_p} h_N^{(i_1 \etc i_p)}$. Hence, a family of random variables $(a_1 \etc a_p)$ as in the proposition exists and its joint distribution has density $\bar h_N$ with respect to the Lebesgue measure on $\mathbb R^p$.
\\
\\{\it Step 3.} We now consider the general case. Let $(X_1 \etc X_p)$ be a family of independent random matrices, independent of $(A_1 \etc A_p)$, distributed according to the standard Gaussian measure on $\mathcal H_N$ with respect to the inner product $\langle A,B\rangle = N \Tr [AB]$. By the regularizing process of convolution on Hermitian space, for any $\eps>0$, the joint distribution of $(A_1^\eps \etc A_p^\eps) = (A_1+\eps X_1 \etc A_p+ \eps X_p)$ has a density with respect to the Lebesgue measure. By the previous step, there exists a family of random variables $(a_1^\eps \etc a_p^\eps)$ such that $\esp\big[ (a_1^\eps)^{n_1} \dots  (a_p^\eps)^{n_p} \big] = \esp\big[  \frac 1 N \Tr \big[ (A_1^\eps)^{n_1} \circ \dots \circ (A_p^\eps)^{n_p} \big] \big]$ for any $n_1 \etc n_p \geq 0$. As $\eps$ goes to zero, $(a_1^\eps \etc a_p^\eps)$  converges in moments to a family of random variables as in the proposition.
\end{proof}

 \begin{thebibliography}{10}
 \bibitem{alice-greg-ofer}
G.~Anderson, A.~Guionnet, O.~Zeitouni
\newblock \emph{An Introduction to Random Matrices}.
\newblock  Cambridge studies in advanced mathematics, {118} (2009).
\bibitem{greg-ofer} 
G. Anderson, O. Zeitouni \emph{A {CLT} for a band matrix model},  Probab. Theory Rel. Fields,
     2005, {\bf 134}, 283--338

 \bibitem{ABP09} A. Auffinger, G. Ben Arous, S. P\'ech\'e \emph{Poisson convergence for the largest eigenvalues of heavy tailed random matrices}. Ann. Inst. Henri Poincar\'e Probab. Stat. 45 (2009), no. 3, 589--610.

\bibitem{bai-silver-book} Z.D.~Bai, J.W.~Silverstein \emph{Spectral analysis of large dimensional random matrices}. Second Edition, Springer, New York, 2009.

\bibitem{baiysilver} Z.D.~Bai, J.~Silverstein \emph{CLT for linear spectral statistics of large-dimensional sample covariance matrices.} Ann. Probab. {\bf 32}, 2004, 533--605.

 \bibitem{BAI2009EJP} Z.D.~Bai,  X. Wang, W. Zhou  \emph{CLT for linear spectral statistics of Wigner matrices.} Electron. J. Probab. {\bf 14} (2009), no. 83, 2391--2417.
 \bibitem{BaiYaoBernoulli2005} Z.D.~Bai, J. Yao  \emph{On the convergence of the spectral empirical process of Wigner matrices.} Bernoulli {\bf 11} (2005) 1059--1092.
 
  \bibitem{BDG} S. Belinschi, A. Dembo,  A. Guionnet   \emph{Spectral measure of heavy tailed band and covariance randommatrices}, {Comm. Math. Phys.}, {\bf 289},  {2009},  1023--1055.
  
     \bibitem{Gerard_Kim} G.~Ben Arous, K.~Dang 
\emph{On fluctuations of eigenvalues of random permutation matrices},
 {arXiv:1106.2108, preprint}.

\bibitem{BAGheavytails}  G. Ben Arous, A. Guionnet \emph{The spectrum of heavy tailed random matrices}. Comm. Math. Phys. {\bf 278} (2008), no. 3, 715--751.

\bibitem{HTBM12} F. Benaych-Georges, S. P\'ech\'e \emph{Localization and delocalization for heavy tailed band matrices}. To appear in Ann. Inst. Henri Poincar\'e Probab. Stat.

\bibitem{Billingsley} P. Billingsley \emph{Probability and measure}, Wiley, third edition.

  \bibitem{Bingham-Goldie-Teugels} N.H. Bingham, C.M. Goldie  J.L. Teugels, \emph{Regular variation}, Cambridge University Press, 1989.

\bibitem{charles_alice} C. Bordenave, A. Guionnet \emph{Localization and delocalization of eigenvectors for heavy-tailed random matrices}, arxiv. 

\bibitem{BCC} C.~ Bordenave, P.~ Caputo,    D.~ Chafa{\"{\i}}
              \emph {Spectrum of large random reversible {M}arkov chains:
              heavy-tailed weights on the complete graph},
  {Ann. Probab.},
{\bf 39},
     {2011}, 1544--1590.
     
     \bibitem{BCC2} C.~ Bordenave, P.~ Caputo,    D.~ Chafa{\"{\i}}
    \emph{Spectrum of non-{H}ermitian heavy tailed random matrices},
  {Comm. Math. Phys.},
 {\bf 307}, {2011}, 513--560.
 
    \bibitem{brezis} H. Brezis \emph{Functional analysis, Sobolev spaces and partial differential equations}, Universitext, Springer (2011).
 
\bibitem{CB} P. Cizeau, J.-P. Bouchaud, \emph{Theory of L\'evy matrices} Phys. Rev. E 50 (1994).

\bibitem{cabanal} T.~Cabanal-Duvillard \emph{Fluctuations de la loi empirique de grande matrices alÈatoires}  Ann. Inst. H.
Poincar\'eÈ Probab. Stat. {\bf 37}, 2001,  373--402.

\bibitem{chatterjee} S.~Chatterjee
     \emph{Fluctuations of eigenvalues and second order {P}oincar\'e
              inequalities},
  {Probab. Theory Related Fields},
 {\bf 143},
     {2009},{1--40}.
     
\bibitem{DEI}
{P. Deift}
\emph{Orthogonal polynomials and random matrices: a
{R}iemann-{H}ilbert approach},
{Courant Lecture Notes in Mathematics},
 {3},
  {New York},
 {1999}.
 
\bibitem{DiEv}
P.~ Diaconis, S. Evans
     \emph{Linear functionals of eigenvalues of random matrices},
 {Trans. Amer. Math. Soc.},
{\bf 353}, {2001}, 2615--2633.

\bibitem{DiSha} P. Diaconis, M. Shahshahani \emph{On the eigenvalues of random matrices},
Studies in applied probability, J. Appl. Probab. 31A (1994), 49--62

    \bibitem{DumiJPP} I.~Dumitriu, T.~ Johnson, S.~Pal, E.~Paquette \emph{Functional limit theorems for random regular graphs}, 
  {Probab. Theory Related Fields},
     {2012},{1--55}.

\bibitem{ESY2} L. Erd\"os, B. Schlein, H.T. Yau \emph{Semicircle law on short scales and delocalization of eigenvectors for Wigner random matrices}, Ann. Prob. 37 (2009).

\bibitem{EYY}{ L.~Erd\"os, H.T.~ Yau,   J.~Yin}
     \emph{Rigidity of eigenvalues of generalized {W}igner matrices},  {Adv. Math.},
{\bf 229},
  {2012}, {1435--1515}.

   \bibitem{feller2} W. Feller \emph{An introduction to probability theory and its applications}, volume II, second edition, New York London Sydney : J. Wiley, 1966.
  
\bibitem{GUI}
 {A. Guionnet},
\emph{Large random matrices: lectures on macroscopic asymptotics},
{Lecture Notes in Mathematics},
 {1957},
 {Lectures from the 36th Probability Summer School held in
             Saint-Flour, 2006},
 {Springer-Verlag},
{Berlin},
 {2009}.
 \bibitem{guionnet} A.~Guionnet \emph{
 Large deviations upper bounds and central limit theorems for non-commutative functionals of Gaussian large random matrices}. Ann. Inst. H. Poincar\'e Probab. Stat. {\bf 38}, 2002,  341--384.
 
 \bibitem{maurel} A.~Guionnet,   E.~Maurel-Segala
   \emph{Second order asymptotics for matrix models}, {Ann. Probab.}, {\bf 35},
  {2007}, 2160--2212.
  
 \bibitem{johansson88} K.~Johansson \emph{On Szeg\"{o} asymptotic formula for Toeplitz determinants
and generalizations},
Bull. des Sciences Math\'ematiques, vol. 112 (1988), 257-304.
  
 \bibitem{johansson} K.~Johansson \emph{On the fluctuations of eigenvalues of random Hermitian matrices.} Duke Math. J. {\bf 91} 1998, 151--204.
 
    \bibitem{jonsson}  D.~Jonsson \emph{Some limit theorems for the eigenvalues of a sample covariance matrix.} J. Mult. Anal. {\bf 12}, 1982, 1--38.


\bibitem{KKP96}  A. M. Khorunzhy, B. A. Khoruzhenko, L. A. Pastur \emph{Asymptotic properties of large random
matrices with independent entries}, {J. Math. Phys.} 37 (1996) 5033--5060.  

\bibitem{KSV04} O. Khorunzhy, M. Shcherbina, V. Vengerovsky
\emph{Eigenvalue distribution of large weighted random graphs},
J. Math. Phys. 45 (2004), no. 4, 1648--1672. 
    
\bibitem{lytova}
A.~Lytova,    L.~ Pastur
     \emph{Central limit theorem for linear eigenvalue statistics of
              random matrices with independent entries},
    {Ann. Probab.}, {\bf 37}, {2009},
   1778--1840.
\bibitem{MAL12}
{C. Male}
\emph{The distribution of traffics and their free product},
 {arXiv:1111.4662v3 preprint}.
 
\bibitem{MAL122}
{C. Male}
\emph{The limiting distributions of large heavy Wigner and arbitrary random matrices},
 {arXiv:1111.4662v3 preprint}.
 
\bibitem{mingo} J.A ~Mingo, R.~Speicher \emph{Second  order freeness and fluctuations of random matrices. I. Gaussian and Wishart matrices and cyclic Fock spaces}. J. Funct. Anal. {\bf 235}, 2006,  226--270.

\bibitem{MF91} A. Mirlin,  Y. Fyodorov \emph{Universality of level correlation function of sparse random matrices}. J. Phys. A 24 (1991), no. 10, 2273--2286.


\bibitem{samotaqqu} G. Samorodnitsky, M. Taqqu  \emph{Stable non-Gaussian random processes. Stochastic models with infinite variance. Stochastic Modeling.} Chapman \& Hall, New York, 1994.

 \bibitem{MShcherbina11} M.~Shcherbina
\emph{Central Limit Theorem for Linear Eigenvalue Statistics of the Wigner and
Sample Covariance Random Matrices}, Journal of Mathematical Physics, Analysis,
Geometry, 7(2), (2011), 176--192.

 \bibitem{tirozzi} M.~Shcherbina, B.~Tirozzi,  \emph{Central limit theorem for fluctuations of linear eigenvalue
              statistics of large random graphs}, {J. Math. Phys.},
  {\bf 51},  {2010}, {023523, 20}

  \bibitem{sinai} Y.~Sinai,   A.~Soshnikov \emph{
     Central limit theorem for traces of large random symmetric
              matrices with independent matrix elements},
 {Bol. Soc. Brasil. Mat. (N.S.)}, {\bf 29}, {1998},
  1--24.
  
  \bibitem{soshni00} A. Soshnikov \emph{The central limit theorem for local linear statistics in
classical compact groups and related combinatorial identities}, Ann. Probab., 28 (2000), 1353--1370.
  
  \bibitem{Tao-Vu_0906.0510} T. Tao, V. Vu \emph{Random matrices: universality of local eigenvalue statistics},  Acta Mathematica, 	
 206 (2011), 127--204.
 
\bibitem{V04} V. Vengerovsky \emph{Asymptotics of the correlator of an ensemble of sparse random matrices}. (Russian) Mat. Fiz. Anal. Geom. 11 (2004), no. 2, 135--160. 

\bibitem{wigner}  E.P. Wigner \emph{On the distribution of the roots of certain symmetric
              matrices},  {Annals Math. },
    {\bf 67},
     {1958},
    {325--327}.
    
    \bibitem{ZAK}
  {I. Zakharevich}
\emph{A generalization of {W}igner's law},
{Comm. Math. Phys.},
 {268},
{2006},
{2},
 {403--414}.

 \en{thebibliography}

\end{document}